\documentclass[11pt,a4paper]{article}
\pdfoutput=0

\usepackage{a4wide,amsfonts,amsmath,latexsym,amsthm,amssymb,euscript,graphicx,units,mathrsfs, color, mathabx}

\usepackage{geometry}
\usepackage{float}
\usepackage{amssymb}
\newfloat{figure}{H}{lof}
\floatname{figure}{\figurename}

\usepackage[T1]{fontenc}

\usepackage[francais,english]{babel}
\usepackage[applemac]{inputenc}

\newcommand{\Norm}[1]{\left\|#1\right\|}
\newtheorem{prop}{Proposition}[section]
\newtheorem{remark}[prop]{Remark}

\newtheorem{theorem}[prop]{Theorem}
\newtheorem{lemma}[prop]{Lemma}

\newtheorem{definition}[prop]{Definition}
\newtheorem{assumption}[prop]{Assumption}
\newtheorem*{conjecture}{Conjecture}

\numberwithin{equation}{section}

\def\E{\mathbb{E}}
\def\P{\mathbb{P}}
\def\Q{\mathbb{Q}}
\def\real{\mathbb{R}}
\def\t{{t \wedge \tau}}
\def\T{{T \wedge \tau}}
\def\S{\mathbb{S}}
\def\H{\mathbb{H}}
\def\L{\mathbb{L}}
\def\1{\textbf{1}}

\def\nequiv{\not\equiv}
\def\F{\mathbb{F}}
\def\G{\mathbb{G}}

\def\de{\delta}

\def\y{\mathcal{Y}}
\def\z{\mathcal{Z}}
\def\u{\mathcal{U}}

\def\cal#1{\mathcal{#1}}

\def\esssup{\mathrm{esssup}}

\begin{document}
\centering \Large Utility maximization with random horizon: a BSDE approach
\vspace{2em}

\large
\begin{tabular}{c c c}
Monique Jeanblanc \footnotemark[1] & & Thibaut Mastrolia \footnotemark[3] \vspace{1em} \\
\small{Universit\'e d'Evry Val d'Essonne} \footnotemark[2] & & \small{Universit\'e Paris Dauphine} \footnotemark[4] \\
\small  LaMME, UMR CNRS 8071    & & \small CEREMADE UMR CNRS 7534\\
\small {\texttt{monique.jeanblanc@univ-evry.fr}}   & & \small{\texttt{thibaut.mastrolia@ceremade.dauphine.fr}}\\
& & \\
&&\\
Dylan Possama\"i & & Anthony R\'eveillac \vspace{1em} \\
\small{Universit\'e Paris Dauphine} \footnotemark[4] && \small INSA de Toulouse \footnotemark[5] \\
\small CEREMADE UMR CNRS 7534& &\small IMT UMR CNRS 5219\\
\small{\texttt{dylan.possamai@ceremade.dauphine.fr}} & & \small Universit\'e de Toulouse\\
& & \small{\texttt{anthony.reveillac@insa-toulouse.fr}} \normalsize
\end{tabular}
\vspace{2em}

\abstract{\noindent In this paper we study a utility maximization problem with random horizon and reduce it to the analysis of a specific BSDE, which we call BSDE with singular coefficients, when the support of the default time is assumed to be bounded. We prove existence and uniqueness of the solution for the equation under interest. Our results are illustrated by numerical simulations.}

\vspace{1em}
{\noindent \textbf{Key words:} quadratic BSDEs, enlargement of filtration, credit risk.
}

\vspace{1em}
\noindent
{\noindent \textbf{AMS 2010 subject classification:} Primary: 60H10, 91G40; Secondary: 91B16, 91G60, 93E20, 60H20.}

\footnotetext[1]{The research of Monique Jeanblanc is supported
by Chaire Markets in transition, French Banking Federation.}
\footnotetext[2]{Laboratoire de Math\'ematiques et
Mod\'elisation d'\'Evry (LaMME),   Universit\'e
d'\'Evry-Val-d'Essonne, UMR CNRS 8071, IBGBI, 23 Boulevard de
France, 91037 Evry Cedex, France} \footnotetext[3]{Thibaut
Mastrolia is grateful to R\'egion Ile de France for financial
support and acknowledges INSA de Toulouse for its warm
hospitality.} \footnotetext[4]{Place du Mar\'echal de Lattre de
Tassigny, 75775 Paris Cedex 16, France} \footnotetext[5]{INSA,
D\'epartement GMM, 135 avenue de Rangueil, F-31077 Toulouse Cedex
4, France}

\section{Introduction}

In recent years, the notion of risk in financial modeling has received a growing interest. One of the most popular direction so far is given by model uncertainty where the parameters of the stochastic processes driving the financial market are assumed to be unknown (usually referred as drift or volatility uncertainty).  Another source of risk consists in an exogenous process which brings uncertainty on the market or on the economy. This kind of situation fits, for instance, in the credit risk theory. As an example, consider an investor who may not be allowed to trade on the market after the realization of some random event, at a random time $\tau$, which is thought to be unpredictable and external to the market. In that context $\tau$ is seen as the time of a shock that affects the market or the agent. More precisely,  assume that an agent initially aims at maximizing her expected utility on a given financial market during a period $[0,T]$, where $T>0$ is a fixed deterministic maturi
 ty.  However, she may not have access to the market after the random time $\tau$. In that context we think of $\tau$ as a death time, either for the agent herself, or for the market (or a specific component of it) she is currently investing in. Though very little studied in the literature, our conviction is that such an assumption can be quite relevant in practice. Indeed, for instance many life-insurance type markets consists of products with very long maturities (up to 95 years for universal life policies and to 120 years for whole life maturity). It is therefore reasonable to consider that during such a period of time an agent in age of investing money in the market will die with probability $1$. Another example is given by markets whose maximal lifetime is finite and known at the beginning of the investment period, like for instance carbon emission markets in the United States.

\vspace{0.5em}
\noindent Mathematically, while her original problem writes down as
\begin{equation}\label{pb:intro}\sup_{\pi \in \mathcal{A}} \E[U(X_T^\pi)],\end{equation}
with $\mathcal{A}$ the set of admissible strategies $\pi$ for the agent with associated wealth process $X^\pi$ and where $U$ is a utility function which models her preferences, due to the risk associated with the presence of $\tau$, her optimization program actually has to be formulated as
\begin{equation}
\label{pb:introdefaut}
\sup_{\pi \in \mathcal{A}} \E[U(X_{\T}^\pi)],
\end{equation}
which falls into the class of \textit{a priori} more complicated stochastic control problems with random horizon.

\vspace{0.5em}
\noindent The main approach to tackle \eqref{pb:introdefaut} consists in rewriting it as a utility maximization problem with deterministic horizon of the form \eqref{pb:intro}, but with an additional consumption component using the following decomposition from \cite{dellacherie} that we recall:
$$ \sup_{\pi \in \mathcal{A}} \E\left[U(X_{\T}^\pi)\right]=  \sup_{\pi \in \mathcal{A}}
\E\left[\int_0^T U(X_{u}^\pi) dF_u +U(X_T^\pi)(1-F_T)\right], $$
with $F_t:=\P(\left.\tau\leq t\right|\mathcal{F}_t)$ and
$\mathbb{F}:=(\mathcal{F}_t)_{t\in [0,T]}$ being the underlying
filtration on the market. This direction was first followed in
\cite{kw} when $\tau$ is a $\mathbb{F}$-stopping time, then in
\cite{BSELKJM} and in \cite{Bouchard_Pham} if $\tau$ is a general
random time. In all these papers, the convex duality theory (see
e.g. \cite{bismut} and \cite{klsx}) is exploited to prove the
existence of an optimal strategy. However, this approach does not
provide a characterization of either the optimal strategy or of
the value function (note that in \cite{BSELKJM} a dynamic
programming equation can be derived if one assumes that $F$ is
deterministic and $U$ is a Constant Relative Risk Aversion
(CRRA) utility function). Another route is to adapt to the random
horizon setting the, by now well-known, methodology in which one
reduces the analysis of a stochastic control problem with fixed
deterministic horizon to the one of a Backward Stochastic
Differential Equation (BSDE) as in \cite{Hu_Imkeller_Mueller,ElKaroui_Rouge}. This program has been
successfully carried out in \cite{KharroubiLimNgoupeyou} in which
Problem \eqref{pb:introdefaut} has been proved to be equivalent to
solving a BSDE with random horizon of the form
\begin{equation}\label{edsr:intro} Y_t = 0 -\int_\t^\T Z_s \cdot dW_s -\int_\t^\T U_s dH_s-\int_\t^\T f(s,Y_s,Z_s,U_s) ds, \  t\in [0,T],
\end{equation}
in the context of mean-variance hedging, with $H_s:=\mathbf{1}_{\tau\leq s}$ and $W$ a standard Brownian motion. The interesting feature here lies in the fact that under some assumptions on the market, the solution triplet $(Y,Z,U)$ to the previous BSDE is completely described in terms of the one of a BSDE with deterministic finite horizon. More precisely, if we assume that $\F$ is the natural filtration of $W$ and if $\tau$ is a random time which is \textit{not} a $\F$-stopping time, then the BSDE with deterministic horizon associated with BSDE \eqref{edsr:intro} is of the form
\begin{equation}
\label{eq:brownienneintro}
Y_t^b = 0 -\int_t^T Z_s^b \cdot dW_s -\int_t^T f^b(s,Y_s^b,Z_s^b) ds, \  t\in [0,T],
\end{equation}
with $f^b$ related to $\tau$ through a predictable process
$\lambda$ (see Section \ref{enlargement} for a precise statement
on this relationship). The usual hypothesis, for instance in
credit risk modeling, is to assume $\lambda$ to be bounded (as in
\cite{KharroubiLimNgoupeyou}). This assumption, which looks pretty
harmless, leads in fact to several consequences both on the
modeling of the problem and on the analysis required to solve
Equation \eqref{edsr:intro}. Indeed, $\lambda$ is bounded implies
that the support\footnote{\textit{i.e.} the smallest closed
Borelian set $A$ such that $\P[\tau \in A]=1$} of $\tau$ is
unbounded. As a consequence, the probability of the event $\{\tau
> T \}$ is positive. Hence it does not take into account the
situation where $\tau$ is smaller than $T$ with probability one.
Note that from the very definition of \eqref{pb:introdefaut},
assuming $\tau$ to have a bounded or an unbounded support leads to
two different economic problems: if the support is unbounded, with
positive probability the agent will be able to invest on the
market up to time $T$, whereas if $\tau$ is known to be smaller
than $T$ with probability one, the agent knows she will not have
access to the market on the whole time interval $[0,T]$.

\vspace{0.5em}
\noindent
The main goal of this paper is to solve \eqref{pb:introdefaut} when the support of $\tau$ is assumed to be a bounded interval in $[0,T]$. As explained in the previous paragraph, this assumption leads to the unboundedness of $\lambda$. More precisely, it generates a singularity in Equation \eqref{edsr:intro} (or in \eqref{eq:brownienneintro}) as $\lambda$ is integrable on any interval $[0,t]$ with $t<T$, and is not integrable on $[0,T]$. This drives one to study a new class of BSDEs, named as BSDEs with singular driver according to \cite{JR}, which requires a specific analysis. We stress that the study of the BSDE of interest of the form \eqref{eq:brownienneintro} with $f^b$ to be specified later is not contained in \cite{JR}, and hence calls for new developments presented in this paper. Incidentally, we propose a unified theory which covers both cases of bounded and unbounded support for $\tau$ (see Conditions $($H2$)$, $($H2'$)$ for a precise statement).

\vspace{0.5em}
\noindent
The rest of this paper is organized as follows. In the next section we provide some preliminaries and notations and make precise the maximization problem under interest. Then in Section \ref{expo}, we extend the results of \cite{Hu_Imkeller_Mueller, KharroubiLimNgoupeyou} allowing to reduce the maximization problem with exponential utility to the study of a Brownian BSDE. The analysis of this equation is done in Section \ref{analyze}. To illustrate our findings, and to compare problems of the form \eqref{pb:intro} and \eqref{pb:introdefaut}, we collect in Section \ref{numerical} numerical simulations together with some discussion.

\vspace{0.5em}
\noindent {\bf Notations:} Let $\mathbb{N}^*:=\mathbb{N}\setminus\{0\}$ and let $\mathbb{R}^+$ be the set of real positive numbers. Throughout this paper, for every $p$-dimensional vector $b$ with $p\in \mathbb{N}^*$, we denote by $b^{1},\ldots,b^{p}$ its coordinates and for $\alpha,\beta \in \real^p$ we denote by $\alpha\cdot \beta$ the usual inner product, with associated norm $\Norm{\cdot}$, which we simplify to $|\cdot|$ when $p$ is equal to $1$. For any $(l,c)\in\mathbb N^*\times\mathbb N^*$, $\mathcal M_{l,c}(\mathbb R)$ will denote the space of $l\times c$ matrices with real entries. When $l=c$, we let $\mathcal M_{l}(\mathbb R):=\mathcal M_{l,l}(\mathbb R)$. For any $M\in\mathcal M_{l,c}(\mathbb R)$, $M^T$ will denote the usual transpose of $M$. For any $x\in\mathbb R^p$, ${\rm diag}(x)\in\mathcal M_{p}(\mathbb R)$ will stand for the matrix whose diagonal is $x$ and for which off-diagonal terms are $0$, and $I_p$ will be the identity of $\mathcal M_p(\mathbb R)$. In this pape
 r the integrals $\int_t^s$ will stand for $\int_{(t,s]}$. For any $d\geq 1$ and for any Borel measurable subset  $I\subset \mathbb R^d$, $\mathcal B(I)$ will denote the Borel $\sigma$-algebra on $I$. Finally, we set for any $p\in \mathbb N^*$, for any closed subset $C$ of $\mathbb R^p$ and for any $a\in\mathbb R^p$
$$\text{dist}_C(a):= \underset{b\in C}{\text{min}}\{\|a-b\|\}, $$
and
$$ \Pi_{C}(a):=\left\{b\in C_t(\omega), \; \|a-b \|=\text{dist}_C(a) \right\}.$$

\section{Preliminaries}\label{prel}

\subsection{The utility maximization problem}

Set $T$ a fixed deterministic positive maturity. Let $W=(W_t)_{t\in [0,T]}$ be a $d$-dimensional Brownian motion ($d\geq 1$) defined on a filtered probability space $(\Omega, \cal G_T,\F,\P)$, where $\F:=(\cal F_t)_{t\in [0,T]}$ denotes the natural completed filtration of $W$, satisfying the \textit{usual conditions}. $\cal G_T$ is a given $\sigma$-field which strictly contains $\cal F_T$ and which will be specified later. Unless otherwise stated, all equalities between random variables on $(\Omega,\mathcal G_T)$ are to be understood to hold $\mathbb P-a.s.$, and all equalities between processes are to be understood to hold $\mathbb P\otimes dt-a.e.$ (and are as usual extended to hold for every $t\geq 0$, $\mathbb P-a.s.$ if the considered processes have trajectories which are, $\mathbb P-a.s.$, c\`adl\`ag\footnote{As usual, we use the french acronym "c\`adl\`ag" for trajectories which are right-continuous and admit left limits, $\P\otimes dt$-a.e.}). The symbol $\mathbb E$ will alwa
 ys correspond to an expectation taken under $\mathbb P$, unless specifically stated otherwise.

\vspace{0.5em}
\noindent We define a financial market with a riskless bond denoted by $S^0:=(S_t^0)_{t\in [0,T]}$  whose dynamics are given as follows:
$$ S_t^0=S_0^0 e^{rt}, \  t\in [0,T], $$
where $r$ is a fixed deterministic non-negative real number. We enforce throughout the paper the condition
$$r:=0,$$
and emphasize that solving the utility maximization problem considered in this paper with a non-zero interest rate is a much more complicated problem.

\vspace{0.5em}
\noindent
Moreover, we  assume that the financial market  contains a $m$-dimensional risky asset $S:=(S_t)_{t\in [0,T]}$ ($1\leq m\leq d$)
$$ S_t=S_0 + \int_0^t{\rm{diag}}(S_s) \sigma_s dW_s + \int_0^t{\rm{diag}}(S_s) b_s ds, \ t \in [0,T].$$
In that setting, $\sigma$ is a $\mathcal M_{m,d}(\mathbb R)$-valued, $\F$-predictable bounded process such that $\sigma \sigma^{T}$ is invertible, and uniformly elliptic\footnote{\textit{i.e.} there exists $K, \varepsilon>0$, s.t. $K I_d \geq \sigma_t \sigma_t^{T} \geq \varepsilon K I_d$,}, $\P\otimes dt-a.e$., and $b$ a $\mathbb R^m$-valued bounded $\F$-predictable process.

\vspace{0.5em}
\noindent We aim at studying the optimal investment problem of a small agent on the above-mentioned financial market with respect to a given utility function $U$ (that is an increasing, strictly concave and real-valued function, defined either on $\real$ or on $\real_+$), but with a random time horizon modeled by a ($\cal G$-measurable) random time $\tau$. More precisely the optimization problem writes down as:
\begin{equation}
\label{eq:pb}
\sup_{\pi \in \cal A} \E[U(X_{T\wedge \tau}^\pi-\xi)],
\end{equation}
where $\cal A$ is the set of admissible strategies which will be
specified depending on the definition of $U$. The wealth process
associated to a strategy $\pi$ is denoted $X^\pi$ (see
\eqref{eq:wealth} below for a precise definition) and $\xi$ is the
liability which is assumed to be bounded, and whose measurability
will be specified later. The important feature of the random time
$\tau$ is that it cannot be explained by the stock process only,
in other words it brings some uncertainty in the model. This can
be mathematically translated into the fact that $\tau$ is assumed
not to be an $\mathbb{F}$-stopping time.

\subsection{Enlargement of filtration}\label{enlargement}

In a general case, $\tau$ can be considered as a default time (see \cite{BJR} for more details). We introduce the right-continuous default indicator process $H$ by setting $$H_t = \1_{\tau\leq t},\ t\geq 0.$$ We therefore use the standard approach of progressive enlargement of filtration by considering $\mathbb{G}$ the smallest right continuous extension of $\mathbb{F}$ that turns $\tau$ into a $\mathbb{G}$-stopping time. More precisely $\mathbb{G}:=(\mathcal{G}_t)_{0 \leq t\leq T}$ is defined by
$$\mathcal{G}_t :=  \bigcap_{\epsilon>0} \tilde {\mathcal{G}}_{t+\epsilon},$$
for all $t\in [0,T]$, where $\tilde{\mathcal{G}}_{s} := \mathcal{F}_s\vee \sigma(H_u\;,u\in[0,s])$, for all $0 \leq s\leq T$.

\vspace{0.5em}
\noindent The following two assumptions on the model we consider will always be, implicitly or explicitly, in force throughout the paper
\begin{itemize}
\item[{\bf (H1)}](Density hypothesis) For any $t$, there exists a
map $\gamma(t, \cdot):\mathbb R^+\longrightarrow  {\mathbb R}^+$,
such that $(t,u)\longmapsto \gamma(t,u)$ is $\cal F_t \otimes \cal
B((0,\infty))$-mesurable and such that
$$\P[\tau >\theta\vert \cal F_t] =\int_\theta^\infty \gamma(t,u) du, \  \theta \in \real_+,$$
and $\gamma(t,u)=\gamma(u,u)\mathbf{1}_{t\geq u}$.
\end{itemize}
Under (H1), we recall that the  "Immersion hypothesis" is
satisfied, that is, any $\mathbb{F}$-martingale is a
$\mathbb{G}$-martingale.

\begin{remark} If instead of considering Assumption $($H1$)$, we had considered the following weaker assumption

\item[{\bf (H1')}] For any $t$, there
exists a map $\gamma(t,\cdot):\mathbb R^+\longrightarrow \mathbb{
R}^+$, such that $(t,u)\longmapsto \gamma(t,u)$ is $\cal F_t
\otimes \cal B((0,\infty))$-mesurable and such that
$$\P[\tau >\theta\vert \cal F_t] =\int_\theta^\infty \gamma(t,u) du, \  \theta \in
\real_+,$$ then, the immersion hypothesis may not be satisfied and in general we can only say that the Brownian motion $W$ is a $\mathbb G$-semimartingale of the form $dW_t=dW^{\mathbb{G}}_t+\mu_tdt$ where
$W^{\mathbb G}$ is a $\mathbb G$-Brownian motion and $\mu_t dt =
\frac{d  \langle \gamma (\cdot, u), W \rangle _t
}{\gamma(t,u)}\vert _{u=\tau}$. Hence, it suffices to write the
dynamics of $S$ as
$$ S_t=S_0 + \int_0^t{\rm{diag}}(S_s) \sigma_s dW_s^{\mathbb{G}} + \int_0^t{\rm{diag}}(S_s) (b_s+ \sigma _s\mu_s) ds, \ t \in [0,T].$$
The difficulty is that there is no general condition to ensure
that $\mu$ is bounded. Nonetheless, if, for instance, we were to assume that there are no arbitrage opportunities on the market and that we restricted our admissible strategies to the ones which are absolutely continuous, then we could prove that $\E[ \int _0^T \|\mu_s\|^2 ds] <+\infty$, which may be enough in order to solve the problem.
\end{remark}

\noindent In both  cases,   the process $H$ admits
an absolutely continuous compensator, i.e., there exists a
non-negative $\mathbb{G}$-predictable process
$\lambda^{\mathbb{G}}$, called the $\mathbb{G}$-intensity, such
that the compensated process $M$ defined by
\begin{equation}\label{eq:defM}
M_t := H_t - \int_0^t \lambda_s^{\mathbb{G}} ds,
\end{equation}
 is a $\mathbb{G}$-martingale.

\vspace{0.5em} \noindent The process
$\lambda^{\mathbb{G}}$ vanishes after $\tau$, and we can write
$\lambda _t^{\mathbb{G}}= \lambda_t^\mathbb{F} \textbf{1}_{t \leq
\tau}$, where
$$\lambda^\mathbb{F}_t= \frac{\gamma(t,t)}{\P(\tau >t
\vert    \mathcal{F}_t)}, $$
is an $\mathbb{F}$-predictable
process, which is called the $\mathbb{F}$-intensity of the process
$H$. Under the density hypothesis, $\tau$ is not an
$\mathbb{F}$-stopping time, and in fact, $\tau$ avoids
$\mathbb{F}$-stopping times and is a totally inaccessible
$\mathbb{G}$-stopping time, see \cite[Corollary
2.2]{ElKaroui_Jeanblanc_Jiao}. From now on, we use a simplified
notation  and write $\lambda:=\lambda^\F$ and set
$$ \Lambda_t:=\int_0^t \lambda_s ds,   t\in [0,T]. $$
Let $\cal T(\F)$ (resp. $\cal T(\G)$) be the set of $\F$-stopping
times (resp. $\G$-stopping times) less or equal to $T$.

\vspace{0.5em}
\noindent In this paper we will work with two different assumptions. The first one corresponds to the case where the support of $\tau$ is unbounded, and the second one refers to the situation where this support is of the form $[0,S]$ with $S\leq T$. In the latter, without loss of generality, we will assume for the sake of simplicity, that $S=T$. More precisely, we will assume that one of the two following conditions is satisfied

\begin{itemize}
\item[{\bf (H2)}]  $\displaystyle \underset{\rho \in \mathcal{T}(\G)}{\esssup}\ \E\left[\left.\int_\rho^T \lambda_s ds\right| \mathcal{G}_{\rho}\right]<+\infty$.

\vspace{0.5em}
\item[{\bf (H2')}]   $\displaystyle \underset{\rho \in \mathcal{T}(\G)}{\esssup}\ \E\left[\left. \int_\rho^t \lambda_s ds\right| \mathcal{G}_\rho\right]<+\infty$ and for all $t<T$ and $\E\left[ \Lambda_T\right]=+\infty$.
\end{itemize}
Under the filtration $\F$, we deduce from the tower property for conditional expectations that
\begin{itemize}
\item {\bf(H2)} $\displaystyle \Rightarrow \underset{\rho \in \mathcal{T}(\F)}{\esssup}\ \E\left[ \left.\int_{\rho}^T \lambda_s ds\right| \mathcal{F}_{\rho}\right]<+\infty $.
\item {\bf (H2')} $\displaystyle \Rightarrow \underset{\rho \in \mathcal{T}(\F)}{\esssup}\ \E\left[\left. \int_{\rho}^t \lambda_s ds\right| \mathcal{F}_\rho\right]<+\infty$ for all $t<T$ and $\E\left[ \Lambda_T\right]=+\infty$.
\end{itemize}

\vspace{0.5em}
\noindent We emphasize that assuming $($H2$)$ or $($H2'$)$ implies in particular that the martingale $M$ is in BMO$(\G)$ (see below for more details), which implies by the well-known energy inequalities (see for instance \cite{Izu}) the existence of moments of any order for $\Lambda$. More precisely, we have for any $p\geq 1$
 \begin{align}
 \label{energy1}  \textbf{(H2)}& \displaystyle \Rightarrow  \E\left[ \left(\int_0^T \lambda_s ds\right)^p\right]<+\infty, \\
 \label{energy2} \textbf{(H2')}& \displaystyle \Rightarrow \E\left[ \left(\int_0^t \lambda_s ds\right)^p\right]<+\infty\text{ for all $t<T$.} \end{align}

\noindent Furthermore, since by \cite[Proposition 4.4]{ElKaroui_Jeanblanc_Jiao}, $\P[\tau>t\vert \cal F_t]=e^{-\Lambda_t}$, for every $t \geq 0$ we have:
\begin{itemize}
\item {\bf(H2)} $ \Rightarrow  \textrm{Supp}(\tau)\supsetneq[0,T],$
\item {\bf(H2')} $\Rightarrow \textrm{Supp}(\tau)=[0,T]$,
\end{itemize}
where Supp denotes the support of the $\G$-stopping time $\tau$.

\vspace{0.5em}
\noindent The previous remark entails in particular that (H2) and (H2') lead to quite different maximization problems. The model under Assumption (H2) is the one which is the most studied in the literature and expresses the fact that with positive probability, the problem \eqref{eq:pb} is the same as the classical maximization problem with terminal time $T$. Naturally, the expectation formulation puts a weight on the {\sl scenarii} which, indeed, lead to the classical framework. Assumption (H2') expresses the fact that with probability $1$ the final horizon is less than $T$ (see Figure \ref{valeur_pn} for an example). This makes the problem completely different since in the first case the agent fears that some random event may happen, whereas in the second case she knows that it is going to happen. As a consequence, these two different assumptions should make some changes in the mathematical analysis. This feature will become quite transparent when solving BSDEs related to the maximi
 zation problem.

\vspace{0.5em}
\noindent For any $m\in\mathbb N^*$, we denote by $\cal P(\F)^m$ (resp. $\cal P(\G)^m$) the set of $\F$ (resp. $\G$)-predictable processes valued in $\real^m$. If $m=1$ we simply write $\cal P(\F)$ for $\cal P(\F)^1$, and the same for $\G$. We recall from \cite[Lemma 4.4]{Jeulin} the decomposition of any $\mathbb{G}$-predictable process $ \psi $,  given by
\begin{equation}\label{Decomposition:Before:After}
\begin{split}
\psi_t = \psi^0_t \textbf{1}_{t \leq \tau} + \psi^1_t(\tau)  \textbf{1}_{ t > \tau }.
\end{split}
\end{equation}
Here the process $\psi^0$ is $\mathbb{F}$-predictable, and for a
given non-negative $u$, the process $\psi^1_{t}(u)$ with $t\geq u,$
is an $\mathbb{F}$-predictable process. Furthermore, for fixed
$t$, the mapping $\psi^1_t(\cdot)$ is
 $\cal F_t \otimes \mathcal{B}([0, \infty))$-measurable. Moreover, if the process $\psi$ is uniformly bounded, then it is possible to choose $\psi^0$ and $\psi^1(.)$ to be bounded.

\vspace{0.5em}
\noindent We introduce the following spaces

\vspace{0.5em}
\noindent $\bullet$ $\displaystyle \S_{\mathbb F}^2:=\left\{Y=(Y_t)_{t\in [0,T]} \in \cal P(\mathbb{F}), \text{ with continuous paths,}\; \E\left[\underset{t\in [0,T]}{\sup} |Y_{t}|^2\right]<+\infty\right\},$

\vspace{0.5em}
\noindent $\bullet$ $\displaystyle \S_{\mathbb G}^2:=\left\{Y=(Y_t)_{t\in [0,T]} \in \cal P(\mathbb{G}), \text{ with c\`adl\`ag paths,} \; \E\left[\underset{t\in [0,T]}{\sup} |Y_{t}|^2\right]<+\infty\right\},$

\vspace{0.5em}
\noindent $\bullet$ $\displaystyle \S^\infty_\mathbb{F}:=\left\{Y=(Y_t)_{t\in [0,T]} \in \cal P(\mathbb{F}), \text{ with continuous paths,}\; \| Y\|_{\mathbb F,\infty}:=\underset{t\in [0,T]}{\sup} |Y_{t}|<+\infty\right\},$

\vspace{0.5em}
\noindent $\bullet$ $\displaystyle \S^\infty_\mathbb{G}:=\left\{Y=(Y_t)_{t\in [0,T]} \in \cal P(\mathbb{G}), \text{ with c\`adl\`ag paths,}\; \| Y\|_{\mathbb G,\infty}:=\underset{t\in [0,T]}{\sup} |Y_{t}|<+\infty\right\},$

\vspace{0.5em}
\noindent $\bullet$ $\displaystyle \H^2_\mathbb{F}:=\left\{Z=(Z_t)_{t\in [0,T]} \in \cal P(\mathbb{F})^d, \; \E\left[ \int_0^T \|Z_s\|^2 ds\right]<+\infty\right\},$

\vspace{0.5em}
\noindent $\bullet$ $\displaystyle \H^2_\mathbb{G}:=\left\{Z=(Z_t)_{t\in [0,T]} \in \cal P(\mathbb{G})^d, \; \E\left[ \int_0^T \|Z_s\|^2 ds\right]<+\infty\right\},$

\vspace{0.5em}
\noindent $\bullet$ $\displaystyle \L^2_\mathbb{G}:=\left\{U=(U_t)_{t\in [0,T]}\in \cal P(\mathbb{G}), \; \E\left[ \int_0^T |U_s|^2 \lambda_s ds\right]<+\infty\right\}.$

\vspace{0.5em}

\noindent In the following, let $Y$ be in $\mathbb S^\infty_\mathbb F $ (resp. $\mathbb S^\infty_\mathbb G $ ), for the sake of simplicity,  we use the notation $\|Y\|_{\infty}:= \| Y\|_{\mathbb F,\infty}$ (resp. $\| Y\|_{\infty}:= \| Y\|_{\mathbb G,\infty}$).
\noindent We conclude this section with a sufficient condition for the stochastic exponential of a c\`adl\`ag martingale to be a true martingale. Given a $\G$-semimartingale $P:=(P_t)_{t\in [0,T]}$, we denote by $\mathcal{E}(P):=(\mathcal{E}(P)_t)_{t\in [0,T]}$ its Dol\'eans-Dade stochastic exponential, defined as usual by:
$$\mathcal{E}(P)_t:=\exp\left(P_t-\frac12 [P^c, P^c]_t\right) \prod_{0<s \leq t} (1+\Delta_s P) \exp\left(-\Delta_s P\right),$$
with $\Delta_s P:=P_s-P_{s-}$ and where $P^c$ denotes the continuous part of $P$.
A c\`adl\`ag $\G$-martingale $P$ is said to be in BMO$(\P,\G)$ if
$$ \| P \|_{{\rm BMO}(\P,\mathbb{G})}^2:=\underset{\rho \in \cal T(\G)}{\esssup}\ \E\left[|P_T-P_{\rho-}|^2\vert \mathcal{G}_\rho\right]< +\infty.$$
For simplicity, we will omit the $\P$-dependence in the space BMO$(\P,\G)$ and will only specify the underlying probability measure if it is different from $\P$.

\begin{prop}\cite[VII.76]{DellacherieMeyer_2}\label{prop:bmosautborne} The jumps of a {\rm BMO}$(\G)$ martingale are bounded. \end{prop}
\noindent The previous proposition together with the definition of a BMO$(\G)$ martingale imply that it is enough for $P$ to be a BMO$(\G)$ martingale, that it has bounded jumps and satisfies:
$$ \underset{\rho \in \cal T(\G)}{\esssup}\ \E[\left.|P_T-P_{\rho}|^2\right| \mathcal{G}_\rho] <+\infty.$$
For the class of BMO$(\G)$ martingale we have the following property.

\begin{prop}\cite[Theorem 2]{Izu}
\label{prop:BMO}
Assume that $P$ is a $\G$ martingale such that there exists $c,\delta>0$ such that $\Delta_\tau P \geq -1+\delta$ and $|\Delta_\tau P|\leq c$, and which satisfies
$$ \underset{\rho \in \cal T(\G)}{\esssup}\ \E[ \langle P \rangle_T -\langle P \rangle_\rho \vert \cal G_\rho]<+\infty.$$
Then $P$ is a {\rm BMO}$(\G)$ martingale and $\cal E(P)$ is a uniformly integrable martingale.
\end{prop}

\noindent We set for $\mathbb{B}\in \{\mathbb{F},\mathbb{G} \}$
$$\H^2_{\rm BMO,\P}(\mathbb{B}):= \left\{N=(N_t)_{t\in[0,T]}\in \H^2(\mathbb{B}),\; \left(\int_0^t N_s dW_s\right)_{t\in[0,T]}\in {\rm BMO}(\mathbb B,\P) \right\},$$
and use the same convention consisting in omitting the $\P$ dependence unless we are working with another probability measure.

\section{Exponential utility function}\label{expo}

We study in this article a "usual" utility function, namely the exponential function, to solve the utility maximization problem \eqref{eq:pb}, which is open in the framework of random time horizon. By open we mean that, even though we have seen that the existence of an optimal strategy for general utility function has been given in \cite{Bouchard_Pham} using a duality approach, we here aim at characterizing both the optimal strategy $\pi^\ast$ and the value function. To that purpose, we combine the martingale optimality principle and the theory of BSDEs with random time horizon. Note that in the classical utility maximization problem with time horizon $T$ this technique has been successfully applied in \cite{ElKaroui_Rouge} in the exponential framework, and in \cite{Hu_Imkeller_Mueller} for the three classical utility functions, that is exponential, power and logarithm.

\vspace{0.5em}
\noindent Recall the maximization problem \eqref{eq:pb}
$$ \sup_{\pi \in \cal A} \E[U(X_{T\wedge \tau}^\pi-\xi)], $$
where $\cal A$ denotes the set of admissible strategies, that is $\G$-predictable processes with some integrability conditions (precise definitions will be given later on), and $\xi$ is a bounded $\mathcal{G}_{T\wedge \tau}$-measurable random variable. At this stage we do not need to make precise these integrability conditions and the exact definition of the wealth process $X^\pi$. Let us simply note that by definition an element $\pi$ of $\cal A$ will satisfy that $\pi \textbf{1}_{(\tau\wedge T,T]}=0$. This condition together with the characterization of $\G$-predictable processes recalled in \eqref{Decomposition:Before:After} entails that $\pi = \tilde{\pi} \textbf{1}_{[0,\tau\wedge T]}$ with $\tilde{\pi}$ a $\F$-predictable process. Hence in our setting the strategies are essentially $\F$-predictable.

\vspace{0.5em}
\noindent We now turn to a suitable decomposition of $\xi$ when $T <\tau$ or $\tau \leq T$.

\begin{lemma}\label{lemma_decomposition}
Let $\xi$ be a $\mathcal{G}_{T\wedge \tau}$-measurable random variable. Then, there exist $\xi^b$ which is $\mathcal{F}_T$-measurable and an $\mathbb{F}$-predictable process $\xi^a$ such that
\begin{equation}\label{eq:F}
\xi=\xi^b \mathbf{1}_{T<\tau} + \xi^a_\tau \mathbf{1}_{\tau \leq T}.
\end{equation}
\end{lemma}

\begin{proof} Let $\xi$ be a $\mathcal{G}_{T\wedge \tau}$-measurable random variable, we have
$$ \xi=\xi \mathbf{1}_{T<\tau} + \xi \mathbf{1}_{\tau \leq T},$$
which can be rewritten as
\begin{equation*}
\xi=\xi^b \mathbf{1}_{T<\tau} + \hat{\xi}^a \mathbf{1}_{\tau \leq T},
\end{equation*}
where $\xi^b$ is an $\mathcal{F}_T$ measurable random variable and $\hat{\xi}^a$ is $\mathcal{G}_\tau$-measurable. According to \cite[Theorem 2.5]{song}, since the assumption $($H1$)$ holds, we get $\mathcal{F}_\tau=\mathcal{G}_\tau$, where we recall that the $\sigma$-field $\mathcal{F}_\tau$ is defined by
$$ \mathcal{F}_{\tau}=\sigma(X_\tau, \; X \text{ is an $\mathbb{F}$-optional process}).$$
Hence, from the definition of $\mathcal{F}_\tau$, we know that there exists an $\mathbb{F}$-optional process denoted by $\xi^a$ such that $\hat{\xi}^a=\xi^a_\tau, \; \mathbb{P}-a.s.$ Since $\mathbb{F}$ is the (augmented) Brownian filtration, any $\mathbb{F}$-optional process is an $\mathbb{F}$-predictable process.
\end{proof}
 \begin{remark}
 In \cite{KharroubiLimNgoupeyou}, the decomposition \eqref{eq:F} was taken as an assumption. However thanks to Lemma \ref{lemma_decomposition}, we know that as long as $\mathbb{F}$ is the augmented Brownian filtration, it always holds true.
 \end{remark}
\vspace{0.5em} \noindent In our framework, the martingale
optimality principle can be expressed as follows (we provide a
proof for the comfort of the reader even though the arguments are
the exact counterpart of the deterministic horizon problem).

\begin{prop}[Martingale optimality principle for the random horizon problem]
\label{prop:mop}
Let $R^\pi:=(R_t^\pi)_{t\in [0,T]}$ be a family of stochastic processes indexed by $\pi \in \mathcal{A}$ such that
\begin{itemize}
\item[$(i)$] $R_\T^\pi = U(X_{\T}^\pi-\xi), \  \forall \pi \in \mathcal{A}$,
\item[$(ii)$] $R_{\cdot \wedge \tau}^\pi$ is a $\G$-supermartingale for every $\pi$ in $\mathcal{A}$,
\item[$(iii)$] $\exists c\in \real, \; R_0^\pi=c, \  \forall \pi \in \mathcal{A}$,
\item[$(iv)$] there exists $\pi^\ast$ in $\mathcal{A}$, such that $R^{\pi^\ast}$ is a $\G$-martingale.
\end{itemize}
Then, $\pi^\ast$ is a solution of the maximization problem \eqref{eq:pb}.
\end{prop}
\begin{proof}
Let $\pi$ in $\mathcal{A}$. Conditions (i)-(iv) immediately imply that
\begin{align*}
\E[U(X_\T^\pi-\xi)]\overset{\rm(i)}{=}\E[R_\T^\pi] \overset{\rm(ii)}{\leq} R_0^\pi \overset{\rm(iii)}{=} R_0^{\pi^\ast} \overset{\rm (iv)}{=} \E[R_\T^{\pi^\ast}] \overset{\rm(i)}{=} \E[U(X_\T^{\pi^\ast}-\xi)],
\end{align*}
which concludes the proof.
\end{proof}

\noindent Note that until now, we have used neither the definition of $\mathcal{A}$ (provided that the expectation $\E[U(X_\T^\pi)]$ is finite) nor the definition of $U$. However, it remains to construct this family of processes $(R^\pi)_{\pi \in \mathcal{A}}$ and this is exactly at this stage that we need to specify both the utility function $U$ and the set of admissible strategies $\mathcal{A}$.  To this end we set:
\begin{equation}
\label{eq:value}
V(x):=\underset{\pi \in \cal A}{\sup}\;\E[U(X_{T\wedge \tau}^\pi-\xi)],
\end{equation}
where $X_{T\wedge \tau}^\pi$ denotes the value at time $T\wedge \tau$ of the wealth process associated to the strategy $\pi \textbf{1}_{[\t,\T]}$ with initial capital $x$ at time $0$, defined below in \eqref{eq:wealth}. This amounts to say that the optimization only holds on the time interval $[\t,\T]$. From now on, we consider the exponential utility function defined as $$U(x)=-\exp(-\alpha x),\ \alpha>0.$$
In that case we parametrize a $\mathbb R^m$-valued strategy $\pi:=(\pi_t)_{t\in [0,T]}$ as the amount of num\'eraire invested in the risky asset $S$ (component-wise) so that the wealth process $X^\pi$ associated to a strategy $\pi$ is defined as:
\begin{equation}\label{eq:wealth}
X_t^\pi= x+\int_0^t \pi_s \cdot \sigma_s dW_s + \int_0^t \pi_s \cdot b_s ds, \ t\in [0,T].
\end{equation}
Note that under our assumption on $\sigma$ (that is $\sigma \sigma^{T}$ is invertible and uniformly elliptic), the introduction of the volatility process does not bring any additional difficulty compared to the case with volatility one. Indeed, as it is well-known, if we set $\theta:= \sigma^T( \sigma \sigma^T)^{-1}b$ and $p:=\sigma^T\pi $, the wealth process becomes
\begin{equation}
X_t^\pi= x+\int_0^t p_s \cdot dW_s + \int_0^t p_s \cdot\theta_s ds=:X_t^p, \ t\in [0,T],
\end{equation}
and a portfolio is described by the process $p$, which is now
$\mathbb R^d$-valued. Let $C:=(C_t)_{t\in [0,T]}$ be a
predictable process with values in the closed subsets of $\mathbb R^d$. As
in \cite{HendersonLiang} we define the set of admissible
strategies by
$$\mathcal{A}:=\left\{p \in \widetilde{\mathcal{A}}, \;  p\in\mathbb H^2_{\text{BMO}({\mathbb{G}})}\right\},$$
with
\begin{align*}
\widetilde{\cal A}:=\Big\{&(p_t)_{t\in [0,T]} \in\mathcal P(\G)^d, \; p_t\in C_t, \;dt\otimes \mathbb{P}-a.e., \; p \textbf{1}_{(\tau \wedge T,T]}=0\Big\}.
\end{align*}

\noindent Since the liability $\xi$ is bounded, according to \cite[Remark 2.1]{HendersonLiang}, optimal strategies corresponding to the utility maximization problem \eqref{eq:pb} coincide with those of \cite{Hu_Imkeller_Mueller}. In order to give a characterization of both the optimal strategy $p^\ast$ and of the value function $V(x)$ defined by \eqref{eq:value}, we combine the martingale optimality principle of Proposition \ref{prop:mop} and the theory of BSDEs with random time horizon.

\begin{theorem}
\label{th:expo}
 Assume that $(H1)$ and $(H2)$ or $(H2')$ hold. Assume that the BSDE

\begin{equation}
\label{eq:BSDEexpo_r0} Y_t = \xi -\int_\t^\T Z_s\cdot dW_s
-\int_\t^\T U_s dH_s - \int_\t^\T f(s,Y_s,Z_s,U_s) ds, \ t\in
[0,T],
\end{equation}
with
\begin{equation}
\label{eq:driver} f(s,\omega,z,u):=
-\frac{\alpha}{2}{\rm dist}^2{\left(z+\frac1\alpha \theta_s,
C_s(\omega)\right)} +z \cdot \theta_s +
\frac{\Norm{\theta_s}^2}{2\alpha } - \lambda_s \frac{e^{\alpha
u}-1}{\alpha},
\end{equation}
where $\rm dist$ denotes the usual Euclidean distance, admits a unique solution $($in the sense of Definition
\ref{Definition:Solution:BSDE:RandomHorizon}$)$ such that $Y$ and
$U$ are uniformly bounded and such that $\int_0^\cdot Z_s \cdot
dW_s + \int_0^\cdot (e^{\alpha U_s} -1) dM_s$ is a {\rm
BMO}$(\G)$-martingale. Then, the family of processes
$$ R^p_t:=-\exp(-\alpha (X_t^p-Y_t)), \ t \in [0,\T], \ p\in \mathcal{A},$$
satisfies $(i)-(iv)$ of Proposition \ref{prop:mop}, so that
$$ V(x)=-\exp(-\alpha (x-Y_0)),$$
and an optimal strategy $p^\ast\in \mathcal{A}$ for the utility maximisation problem \eqref{eq:value} is given by
\begin{equation}\label{strat:opt}
p_t^\ast \in \Pi_{C_t(\omega)}\left(Z_t+\frac{\theta_t}{\alpha}\right), \ t\in [0,T], \ \mathbb{P}-a.s.
\end{equation}
\end{theorem}

\begin{proof}
Assume that the BSDE \eqref{eq:BSDEexpo_r0} admits a unique solution (in the sense of Definition \ref{Definition:Solution:BSDE:RandomHorizon}) such that $Y$ and $U$ are uniformly bounded and such that $$P:=\int_0^\cdot Z_s \cdot dW_s + \int_0^\cdot (e^{\alpha U_s} -1) dM_s,\ \text{is a BMO$(\G)$ martingale.}$$
Following the initial computations of \cite{Hu_Imkeller_Mueller} (see also \cite{Becherer,Morlais_jump} for the discontinuous case) we set:
$$ R_t^p:=-\exp(-\alpha (X_t^p-Y_t)), \ t \in [0,\T], \ p\in \mathcal{A}. $$ Clearly, the family of processes $R^p$ satisfies Properties (i) and (iii).
By definition each process $R^p$ reduces to
$$ R_t^p=L_t^p \exp\left(\int_0^t v(s,p_s,Z_s,U_s) ds\right),$$
with $$v(s,p,z,u):=\frac{\alpha^2}{2} \Norm{p -z}^2 -\alpha p\cdot \theta + \lambda_s(e^{\alpha u}-1-\alpha u) + \alpha \textbf{1}_{\{s\leq \tau \}} f(s,z,u),$$ and
$$L_t^p:=-\exp(-\alpha(x-Y_0)) \mathcal{E}\left(-\alpha \int_0^\cdot (p_s-Z_s) \cdot dW_s + \int_0^\cdot (e^{\alpha U_s} -1) dM_s\right)_t,$$ which is a uniformly integrable martingale by Proposition \ref{prop:BMO}. As in \cite{Hu_Imkeller_Mueller}, the latter property together with the boundedness of $Y$ and the notion of admissibility for the strategies $p$ imply that each process $R^p$ is a $\G$-supermartingale and that $R^{p^\ast}$ is a $\G$-martingale with $p_t^\ast\in\Pi_{C_t(\omega)}\left(Z_t+\frac{\theta_t}{\alpha}\right)$, $t\in [0,T]$. We conclude with Proposition \ref{prop:mop}.
\end{proof}

\begin{remark}
In this paper we have considered exponential utility, however the case of power utility and/or logarithmic utility follows the same line as soon as $\xi=0$.
\end{remark}

\noindent Of course, the above theorem is a verification type result, which is crucially based on the wellposedness of the BSDE \eqref{eq:BSDEexpo_r0}. We have therefore reduced the analysis of the maximization problem to the study of the BSDE \eqref{eq:BSDEexpo_r0}, which is the purpose of the next section.

\section{Analysis of the BSDE \eqref{eq:BSDEexpo_r0}}\label{analyze}

\subsection{Some general results on BSDEs with random horizon}

As we have seen in the previous section, solving the optimal portfolio problem under exponential preferences (with interest rate $0$) reduces to solving a BSDE with a random time horizon. This class of equations has been studied in \cite{Darling_Pardoux}, and one could construct a classical theory for these equations. However, in our setting the filtration $\G$ is strongly determined by the terminal time $\tau$, and the structure of predictable processes with respect to $\G$ is richer than in the general framework. More precisely, from \cite{Jeulin} we know that a $\G$-predictable process can be described using $\F$-predictable processes before and after $\tau$ as recalled in \eqref{Decomposition:Before:After}.

\vspace{0.5em}
\noindent Recall that by \eqref{eq:F}, any bounded $\mathcal{G}_{T\wedge\tau}$-measurable random variable $\xi$ can be written as
\begin{eqnarray*}
\xi = \xi^b\textbf{1}_{T<\tau} + \xi^a_{\tau}\textbf{1}_{\tau \leq T},
\end{eqnarray*}
with $\xi^b$ a $\mathcal{F}_{T }$-measurable bounded random
variable, and $\xi^a $ a bounded $\F$-predictable process.

\vspace{0.5em} \noindent Taking advantage of this
decomposition, the solution triple to a BSDE with random horizon
$\tau$ has been determined in \cite{KharroubiLimNgoupeyou} as the
one of a BSDE in the Brownian filtration $\F$ suitably stopped at
$\tau$ (see
\eqref{supportS:eq:sol:Y:H2}-\eqref{supportS:eq:sol:U:H2} below
for a precise statement). However we would like to stress that
this result has been obtained under the assumption that $\lambda$
is bounded which is a stronger assumption than (H2).

\vspace{0.5em}
\noindent We consider a BSDE with random terminal horizon of the form
\begin{eqnarray}
\label{BSDE:Y:General:Case} Y_t = \xi-\int_{t\wedge\tau}^{T\wedge\tau}f(s, Y_s, Z_s, U_s)ds-\int_{t\wedge\tau}^{T\wedge\tau}Z_s\cdot dW_s-\int_{t\wedge\tau}^{T\wedge\tau}U_sdH_s.
\end{eqnarray}
From \eqref{Decomposition:Before:After} (see also (4.28) in \cite{KharroubiLimNgoupeyou}), we can write
\begin{eqnarray}
\label{supportS:eq:diff:f}  f(t,.)\textbf{1}_{t<\tau} = f^b(t,.)\textbf{1}_{t<\tau},
\end{eqnarray}
where $f^b:\Omega\times[0,T]
\times\real\times\real^d\times\real\longrightarrow\real$ is
$\mathbb F$-progressively measurable.
\begin{definition}\label{Definition:Solution:BSDE:RandomHorizon}
A triplet of processes $(Y,Z,U)$ in $\S^2_\G \times \mathbb{H}^2_{\G} \times \mathbb{L}^2_\G$ is a solution of the BSDE \eqref{BSDE:Y:General:Case} if relation \eqref{BSDE:Y:General:Case} is satisfied for every $t$ in $[0,\T]$, $\P$-a.s., $Y_t=Y_\T$, for $t\geq \T$, $Z_t=0$, $U_t=0$ for $t > \T$ on the set $\{\tau<T\}$, and
\begin{equation}
\label{eq:condsol}
\E\left[ \int_0^\T |f(t,Y_t,Z_t,U_t)|dt + \left(\int_0^\T \Norm{Z_t}^2 dt\right)^{1/2} \right]<+\infty.
\end{equation}
\end{definition}

\begin{remark}
If $f$ is Lipschitz continuous then the fact that $(Y,Z,U)$ are in the space $\S^2_\G \times \H^2_{\G} \times \L^2_\G$ implies that \eqref{eq:condsol} holds. However under $($H2$)$ or $($H2'$)$, $f$ in \eqref{eq:driver} is not Lipschitz continuous and the fact that $(Y,Z,U)$ are in the space $\S^2_\G \times \H^2_{\G} \times \L^2_\G$ does not guarantee that
$$\E\left[ \int_0^\T |f(t,Y_t,Z_t,U_t)|dt  \right]<+\infty. $$
 \end{remark}

\begin{remark}
Note that the term $\int_0^t U_s dH_s$ is well-defined since it reduces to $U_{\tau} \textbf{1}_{t\geq \tau}$. Another formulation of a solution would consist in re-writing \eqref{BSDE:Y:General:Case} as:
$$ Y_t = \xi-\int_{t\wedge\tau}^{T\wedge\tau}[f(s, Y_s, Z_s, U_s)+\lambda_s U_s]ds-\int_{t\wedge\tau}^{T\wedge\tau}Z_s\cdot dW_s-\int_{t\wedge\tau}^{T\wedge\tau}U_s dM_s, t\in [0,T].$$
In this case, the integrability condition on the driver basically amounts to ask $$\E\left[\int_0^T \lambda_s |U_s| ds\right]<+\infty,$$ which insures that the process $U$ is locally square integrable\footnote{Consider $\rho_n:=\inf\{t\geq \rho_{n-1}, \; |U_t| \geq n\}$ and $\tau_0:=0$, and remark that $\int_0^{\rho_n} |U_s|^2 \lambda_s ds = \int_0^{\rho_n-} |U_s|^2 \lambda_s ds \leq n \int_0^{T} |U_s| \lambda_s ds<\infty, \; \P-$a.s.}, justifying the definition of the stochastic integral $\int_0^\cdot U_s dM_s$.
\end{remark}
\noindent Similarly given $\xi$ an $\cal F_T$-measurable map,
and $f:\Omega\times[0,T]\times \real \times \real^d
\longrightarrow \real$ an $\F$-progressively measurable
mapping, we say that a pair of $\F$-adapted processes $(Y,Z)$ where $Z$ is predictable is a solution of the Brownian BSDE:
\begin{eqnarray}
\label{eq:BSDEb}    Y_t = \xi-\int_{t}^{T} f(s, Y_s, Z_s)ds-\int_{t}^{T}Z_s \cdot dW_s, \  t\in [0,T],
\end{eqnarray}
if Relation \eqref{eq:BSDEb} is satisfied and if

\begin{equation}\label{def:sol:broh3'} \E\left[\int_0^T |f(t,Y_t,Z_t)| dt +\left(\int_0^T \Norm{Z_t}^2 dt\right)^{1/2}\right]<+\infty.\end{equation}

\noindent We recall the following proposition which has been proved in \cite{KharroubiLimNgoupeyou}.
\begin{prop}
\label{propBroH2}
Assume $(H1)$-$(H2)$. If the $($Brownian$)$ BSDE
\begin{equation}
\label{eq:Broexpo1}
Y_t^{b} = \xi^b-\int_{t}^{T}f^b(s, Y_s^{b}, Z_s^{b}, \xi^a_s-Y_s^{b})ds-\int_{t}^{T}Z_s^{b}\cdot dW_s, ~t\in[0, T],
\end{equation}
admits a solution $(Y^b, Z^b)$ in $\S_{\mathbb{F}}^\infty\times\H_{\mathbb{F}}^2$, then $(Y,Z,U)$ defined as
\begin{align}
\label{supportS:eq:sol:Y:H2}    Y_t &= Y_t^{b} \textbf{1}_{t<\tau}+ \xi^a_{\tau}\textbf{1}_{t\geq\tau},\\
\label{supportS:eq:sol:Z:H2}    Z_t &= Z_t^{b} \textbf{1}_{t\leq\tau},\\
\label{supportS:eq:sol:U:H2}    U_t &= (\xi^a_t-Y_t^{b}) \textbf{1}_{t\leq\tau},
\end{align}
is a solution of the BSDE \eqref{BSDE:Y:General:Case} in $\S_{\mathbb{G}}^\infty\times\H_{\mathbb{G}}^2\times\L_{\mathbb{G}}^2$.
\end{prop}

\noindent The previous proposition is in fact a slight generalization of the original result in \cite{KharroubiLimNgoupeyou}, since in this reference the authors assume $\lambda$ to be bounded, which implies condition (H2). In addition, the authors in this reference work with classical solutions in $\S_{\mathbb{G}}^\infty\times\H_{\mathbb{G}}^2\times\L_{\mathbb{G}}^2$. However, the proof follows the same lines as the original proof in \cite{KharroubiLimNgoupeyou}, we just notice that \cite[Step 1 and Step 2 of the proof of Theorem 4.3]{KharroubiLimNgoupeyou} are unchanged and Step 3 holds under Assumption (H2) noticing that
$$ \| U \|_{\L_{\mathbb{G}}^2}^2\leq C \E[\Lambda_{T\wedge \tau}]<+\infty, $$ since $Y^b$ and $\xi^a$ are in $\S_{\mathbb{F}}^\infty$.

\begin{prop}
\label{prop:BroH2'}
We assume $(H1)$ and $(H2')$. Let $A$ be a real-valued, $\mathcal F_T$-measurable random variable such that $\mathbb E[|A|^2]<+\infty$. Assume that the BSDE
\begin{equation}
\label{eq:brobor}
Y_t^{b} = A-\int_{t}^{T}f^b(s, Y_s^{b}, Z_s^{b}, \xi^a_s-Y_s^{b})ds-\int_{t}^{T}Z_s^{b}\cdot dW_s, ~t\in[0, T],
\end{equation}
admits a solution $(Y^b, Z^b)$  in $\S_{\mathbb{F}}^2\times\H_{\mathbb{F}}^2$. Then $(Y, Z, U)$ given by
\begin{eqnarray}
\label{supportS:eq:sol:Y}   Y_t &=& Y_t^{b} \textbf{1}_{t<\tau} + \xi^a_{\tau}\textbf{1}_{t\geq\tau},\\
\label{supportS:eq:sol:Z}   Z_t &=& Z_t^{b} \textbf{1}_{t\leq\tau},\\
\label{supportS:eq:sol:U}   U_t &=& (\xi^a_t-Y_t^{b}) \textbf{1}_{t\leq\tau},
\end{eqnarray}
is a solution of \eqref{BSDE:Y:General:Case} and $(Y,Z,U)$ belongs to $\S_{\mathbb{G}}^2\times\H_{\mathbb{G}}^2\times \S_{\mathbb{G}}^2$.
\end{prop}
\begin{proof}
We reproduce the proof of \cite[Theorem 4.3]{KharroubiLimNgoupeyou}. Step 1 and Step 2 are unchanged and prove that for all $t\in [0,T]$, $(Y,Z,U)$ defined by  \eqref{supportS:eq:sol:Y}, \eqref{supportS:eq:sol:Z} and \eqref{supportS:eq:sol:U} satisfied BSDE \eqref{BSDE:Y:General:Case}. From the definition of $Y$, since $Y^b$ and $\xi^a$ are in $\S_{\mathbb{F}}^2$ we deduce that $Y\in \S_{\mathbb{G}}^2$. from the definition of $Z$, we deduce that $Z\in \H_{\mathbb{G}}^2$.\end{proof}

\begin{remark}
 Note that in the previous result, the fact that $Y^b$ is for example bounded would not imply that $U$ is in $\L_{\mathbb{G}}^2$ as $\lambda$ is not integrable.
\end{remark}

\begin{remark}
The previous result is very misleading since the terminal condition $A$ in \eqref{eq:brobor} plays no role. More precisely, assume that for two different random variables $A^1$ and $A^2$ such that the associated solutions $(Y^{A^1},Z^{A^1},U^{A^1})$ and $(Y^{A^2},Z^{A^2},U^{A^2})$  are bounded and verify that $$\int_0^\cdot Z_s^{A^i} \cdot dW_s + \int_0^\cdot (e^{\alpha U_s{A^i}} -1) dM_s \text{ is a ${\rm BMO}(\G)$-martingale  $(i=1,2)$}.$$ Then obviously $Y^{A^1} \nequiv Y^{A^2}$, and in light of the proof of Theorem \ref{th:expo}, the maximization problem \eqref{eq:value} would then be ill-posed as it would have two different value functions. Even though the notion of strategy we use slightly differs from the one used in \cite{Bouchard_Pham}, this conclusion seems to contradict the well-posedness result obtained in this reference. This remark suggests that it might be possible to solve the Brownian BSDE \eqref{eq:brobor} for only one element $A$. For instance, in the exponential u
 tility setting, Relation \eqref{def:sol:broh3'} suggests that $A\equiv \xi^a_T$ to solve BSDE \eqref{eq:brobor}. To illustrate this, we assume that $\alpha=1$ and that there is no Brownian part. We consider the following Cauchy-Lipschitz/Picard-Lindel\"of problem:
 $$y'_t=\lambda_t (e^{\xi_t^a -y_t}-1), \quad y_T=A.$$ Assume that $\xi^a$ is deterministic, bounded and continuously differentiable. Set $x_t:=e^{y_t}$. Hence, the previous ODE can be rewritten:
  $$ x'_t=\lambda_t(e^{\xi_t^a}-x_t), \quad x_T=e^A. $$

\noindent Thus, we can compute explicitly the unique $($global$)$ solution, which is
$$x_t= e^{-\Lambda_t} C + e^{-\Lambda_t}\int_0^t e^{\xi_s^a}\lambda_s e^{\Lambda_s} ds, \; t\in [0,T], $$
where $C$ is in $\mathbb{R}$. Using an integration by part, one gets
$$ x_t= e^{-\Lambda_t} C +e^{\xi_t^a}-  \int_0^t (\xi^a_s)' e^{\xi_s^a}e^{-\int_s^t \lambda_u du} ds, \; t\in [0,T].$$
Letting $t$ go to $T$, we obtain that we must have $x_T=e^{\xi_T^a}$. Therefore there is a solution if and only if $A=\xi^a_T$.
 \end{remark}
\subsection{BSDEs for the utility maximization problem}

In this section we focus our attention on a class of BSDEs with quadratic growth, which contains in particular the one used for solving the exponential utility maximization problem. We assume that the generator $f$ of BSDE \eqref{BSDE:Y:General:Case} admits for all $(t,\omega,y,z,u)$ in $[0,T]\times\Omega \times \mathbb{R}\times \mathbb{R}^d\times \mathbb{R}$ the following decomposition
\begin{equation}
\label{decomposition:f}f(t,\omega,y,z,u)=g(t,\omega,y,z)+\lambda_t(\omega)\frac{1-e^{\alpha u}}{\alpha},
\end{equation} where $g$ is a map from $[0,T]\times\Omega\times  \mathbb{R}\times \mathbb{R}^d$ to $\mathbb{R}$. We assume moreover that $g$ satisfies
\begin{assumption}\label{assumpt:g}
\begin{itemize}
\item[\textbf{$(i)$}] For every $(y,z)\in  \mathbb{R}\times \mathbb{R}^d$, $g(\cdot,y,z)$ is $\G$-progressively measurable.
\item[\textbf{$(ii)$}] There exists $M>0$ such that for every $t\in [0,T]$, $ |g(t,0,0)|\leq M$,
and for every $(t,\omega,y,y',z,z')\in [0,T] \times \Omega\times \mathbb{R}\times \mathbb{R}\times \mathbb{R}^d\times \mathbb{R}^d$,
$$|g(t,\omega,y,z)-g(t,\omega,y',z)| \leq  M| y-y'|,$$
and
$$|g(t,\omega,y,z)-g(t,\omega,y,z')| \leq M(1+\|z\|+\|z'\|)\|z-z'\|.$$
\end{itemize}
\end{assumption}

\vspace{0.5em}
\noindent Before going further, notice that under Assumption \ref{assumpt:g}, we have the following useful linearization for all $t\in [0,T]$
\begin{equation}
\label{decomposition:g}g(t,\omega,y,z)-g(t,\omega,y',z')=m(t,\omega,y,y') (y-y')+\eta(t,\omega,z,z') \cdot(z-z'),\; \mathbb{P}-\text{a.s.,}
\end{equation} where $m:[0,T]\times \Omega\times \mathbb{R}\times\mathbb{R}\longrightarrow \mathbb{R}$ is $\mathbb G$-progressively measurable and such that $|m(t,y,y')|\leq M$ and $\eta:[0,T]\times\Omega\times \mathbb{R}^d\times \mathbb{R}^d\longrightarrow \mathbb{R}^d$ is $\mathbb G$-progressively measurable and such that
$$\| \eta(t,z,z')\|\leq M(1+\|z\|+\|z'\|), \; \P-\text{a.s}. $$
For simplicity, we will write $\eta(t,z)$ instead of $\eta(t,z,0)$ and $m(t,y)$ instead of $m(t,y,0)$. Notice that under Assumption \ref{assumpt:g}, there exists $\mu>0$ such that for every $t\in [0,T]$ and $y,z \in \mathbb{R}\times \mathbb{R}^d$
$$ |g(t,y,z)| \leq \mu(1+|y|+\|z\|^2), \; \P-\text{a.s.}$$

\subsubsection{A uniqueness result}
We start with a uniqueness result for BSDE \eqref{BSDE:Y:General:Case} under the Assumption \ref{assumpt:g}.
\begin{lemma}
\label{lemme:expounique}
Assume that $(H1)$ and Assumption \ref{assumpt:g} hold. Under $(H2)$ or $(H2')$, the BSDE \eqref{BSDE:Y:General:Case}:
$$ Y_t = \xi -\int_\t^\T Z_s \cdot dW_s -\int_\t^\T U_s dH_s - \int_\t^\T f(s,Y_s,Z_s,U_s) ds, \  t\in [0,T] $$
admits at most one solution $(Y,Z,U)$ such that $$Y \in\S^\infty_\G\text{ and }\int_0^\cdot Z_s\cdot dW_s+\int_0^\cdot (e^{\alpha U_s}-1) dM_s\text{ is a {\rm BMO}$(\G)$ martingale}.$$
\end{lemma}

\begin{remark}
From the orthogonality of $W$ and $M$, notice that
\begin{align*}
& \int_0^\cdot Z_s\cdot dW_s+\int_0^\cdot (e^{\alpha U_s}-1) dM_s \text{ is a {\rm BMO}$(\G)$ martingale}\\\Longleftrightarrow&   \int_0^\cdot Z_s\cdot dW_s \text{ and }\int_0^\cdot (e^{\alpha U_s}-1) dM_s \text{ are two {\rm BMO}$(\G)$ martingales}.
 \end{align*}
\end{remark}
\begin{proof}[Proof of Lemma \ref{lemme:expounique}]
Let $(\mathcal{Y},\mathcal{Z},\mathcal{U})$ and $(\widetilde{\mathcal{Y}},\widetilde{\mathcal{Z}},\widetilde{\mathcal{U}})$ be two solutions of BSDE \eqref{BSDE:Y:General:Case} above with $(\mathcal{Y}, \widetilde{\mathcal{Y}}) \in \mathbb{S}^{\infty}_{\mathbb{G}}\times\mathbb{S}^{\infty}_{\mathbb{G}}$ and such that
$$\int_0^\cdot \z_s\cdot dW_s+\int_0^\cdot (e^{\alpha \u_s}-1) dM_s \  \textrm{ and } \  \int_0^\cdot \widetilde\z_s\cdot  dW_s+\int_0^\cdot (e^{\alpha \widetilde\u_s}-1) dM_s,$$
are two BMO$(\G)$ martingales. Then $(\de \y:=\y-\widetilde\y,\de Z:=\z-\widetilde\z,\de U:=\u-\widetilde\u)$ solves the BSDE:
$$ \de \y_t =0-\int_\t^\T \de \z_s \cdot dW_s- \int_\t^\T \de \u_s dH_s - \int_\t^\T \de f(s) ds, \  t\in [0,T],$$ where $$\de f(s):= g(s,\y_s,\z_s)-g(s,\widetilde\y_s,\widetilde\z_s)- \lambda_s\frac{e^{\alpha\u_s}-e^{\alpha\widetilde\u_s}}{\alpha}. $$
The equation linearizes to obtain
\begin{align*}
 \de \y_t =&\ 0- \int_\t^\T \de \y_s m(s,\mathcal{Y}_s,\widehat\y_s)+ \de \z_s \cdot\eta(s,\z_s,\widehat\z_s) -\lambda_s e^{\alpha \widehat{\u}_s} \de \u_s ds\\
 &-\int_\t^\T \de \z_s \cdot dW_s- \int_\t^\T \de \u_s dH_s,\  t\in [0,T],
 \end{align*}
where $\widehat{\u}_s$ is a point between $\u_s$ and $\widetilde\u_s$, $m$ and $\eta$ are given by Relation \eqref{decomposition:g}. Knowing that $\int_0^\cdot \z_s \cdot dW_s$ and $\int_0^\cdot \widetilde\z_s \cdot dW_s$ are two BMO$(\G)$-martingales, from Assumption \ref{assumpt:g}$(ii)$ we deduce that $\int_0^\cdot \eta(s,\z_s,\tilde \z_s)\cdot dW_s$ is a BMO$(\G)$-martingale and the previous relation re-writes again as:
\begin{equation}
\label{eq:uni1}
\de \y_t =0-\int_\t^\T \de \z_s \cdot dW_s^\Q- \int_\t^\T \de \u_s dM_s^\Q-\int_\t^\T \de \y_s m_sds, \  t\in [0,T],
\end{equation}
with $$\frac{d\Q}{d\P}:=\cal E\left(-\int_0^\cdot\eta(s,\z_s,\widehat\z_s)\cdot dW_s +\int_0^\cdot (e^{\alpha \widehat{ \u}_s}-1) dM_s\right)_T,$$
and $W^\Q:=W+\int_0^\cdot \eta(s,\z_s,\widehat\z_s) \cdot dW_s$ and $M^\Q:=M-\int_0^\cdot (e^{\alpha \widehat{\u}_s}-1) \lambda_s ds$. Note that $\Q$ is a well-defined probability measure, as soon as $\cal E(P)$ with
$$P:=-\int_0^\cdot \eta(s,\z_s,\widehat\z_s)\cdot dW_s +\int_0^\cdot (e^{\alpha \widehat{\u_s}}-1) dM_s,$$ is a true martingale. In that case, the conclusion of the lemma follows by linearization and taking the $\Q$-conditional expectation in \eqref{eq:uni1} knowing that $m$ is bounded. It then remains to prove that the process $P$ is a BMO$(\G)$ martingale which will imply that its stochastic exponential is a uniformly integrable martingale by Proposition \ref{prop:BMO}. Note that since $\int_0^\cdot (e^{\alpha \u_s}-1) dM_s$ and $\int_0^\cdot (e^{\alpha \widetilde\u_s}-1) dM_s$ are two BMO$(\G)$ martingales, then according to Proposition \ref{prop:bmosautborne}, $\u_\tau$ and $\widetilde\u_\tau$ are bounded, hence $\widehat \u_\tau$ is bounded by $c>0$. We deduce that the jump of $P$ at time $\tau$ is bounded and greater than $-1+\delta$ with $\delta:=e^{-\alpha c}>0$.
Since $\widehat{\u}_s$ is an element between $\u_s$ and $\widetilde\u_s$, it is a (random) convex combination of $\u_s$ and $\widetilde\u_s$. The convexity of the mapping $x\mapsto |e^{\alpha x}-1|^2$ implies for any element $\rho$ in $\cal T(\G)$ that
\begin{align*}
&\int_\rho^T |e^{\alpha \widehat{\u}_s}-1|^2 \lambda_s ds\leq C \left( \int_\rho^T |e^{\alpha \u_s}-1|^2 \lambda_s ds + \int_\rho^T |e^{\alpha \widetilde\u_s}-1|^2 \lambda_s ds \right).
\end{align*}
This estimate together with the BMO properties proved so far, imply that $P$ is a BMO$(\G)$ martingale.
\end{proof}

\subsubsection{Existence results for Brownian BSDEs} We turn to the existence of a solution $(Y,Z)$ to the BSDE \eqref{BSDE:Y:General:Case} such that $Y$ is in $\S^\infty_\G$ and $\int_0^\cdot Z_s\cdot dW_s+\int_0^\cdot (e^{\alpha U_s}-1) dM_s$ is a BMO$(\G)$ martingale under Assumptions $($H2$)$ and $($H2'$)$. From Proposition \ref{propBroH2} and Proposition \ref{prop:BroH2'}, this BSDE can be reduced to the following Brownian BSDE

\begin{equation}\label{edsr:brogb}Y_t^b=\xi^b-\int_t^T g^b(s,Y^{b}_s,Z^{b}_s)+\lambda_s\frac{1-e^{\alpha(\xi^a_s-Y_s^{b})}}{\alpha}ds -\int_t^T Z^{b}_s \cdot dW_s, \end{equation}
where $g^b$ satisfies Assumption \ref{assumpt:g} (changing in $(i)$ $\G$-progressively measurable by $\F$-progressively measurable) and inherits the decomposition \eqref{decomposition:g} from the one of $g$ as
\begin{equation}\label{decomposition:gb}g^b(t,\omega,y,z)-g^b(t,\omega,y',z')=m^b(t,\omega,y,y')(y-y')+\eta^b(t,\omega,z,z')(z-z'),\end{equation} for any $(t,y,y',z,z') \in [0,T]\times \real^2\times (\real^d)^2$ with $m^b(t,\cdot):=m(t,\cdot)\mathbf{1}_{t\leq \tau}$ and $\eta^b(t,\cdot):=\eta(t,\cdot)\mathbf{1}_{t\leq \tau}$. However, neither Assumption $($H2$)$ nor Assumption $($H2'$)$ guarantee directly that this quadratic BSDE admits a solution. Hence, we use approximation arguments and introduce quadratic BSDEs defined for $n\geq 1$ by
\begin{equation}\label{bsde:approx}
Y_t^{b,n}=\xi^b-\int_{t}^{T}   g^b(s,Y^{b,n}_s,Z^{b,n}_s) +\lambda_s^n  \frac{1-e^{\alpha (\xi^a_s-Y_s^{b,n})}}{\alpha} ds-\int_{t}^{T}Z_s^{b,n} \cdot dW_s, ~t\in[0, T],\end{equation} where $\lambda^n := \lambda \wedge n$.
By developing the integrand in this BSDE \eqref{bsde:approx}, one obtains
\begin{equation}\label{bsde:approxbis} Y_t^{b,n}=\xi^b-\int_{t}^{T}   g^b(s,Y^{b,n}_s,Z_s^{b,n})-\tilde\lambda_s^{b,n} \xi_s^a +\tilde\lambda_s^n Y_s^{b,n}  ds-\int_{t}^{T}Z_s^{b,n} \cdot dW_s, ~t\in[0, T],\end{equation}
where $\tilde\lambda_s^n:= \lambda_s^n\int_0^1 e^{-\alpha\theta (Y_s^{b,n}-\xi_s^a)}d\theta$.

\begin{lemma}[General {\sl a priori} estimates under (H2)]\label{lem:aprioriestimates}
Let $n\geq 0$. Under Assumptions $(H1)$-$(H2)$ and Assumption \ref{assumpt:g}, the BSDE \eqref{bsde:approx} admits a unique solution $(Y^{b,n},Z^{b,n})\in \S^\infty_{\F} \times \H^2_{{\rm{BMO}}(\F)}$ such that for all $t\in [0,T]$,  $$ |Y_t^{b,n} |\leq e^{M T}\left(\| \xi^b\|_\infty+M (T-t)+\| \xi^a\|_\infty \right)=:C_Y,$$
and $\|Z^{b,n}\|_{\H^2_{\rm BMO}(\F)}$ is uniformly bounded in $n$.
\end{lemma}
\begin{proof} Let $t\in [0,T]$. The proof is divided in several steps.

\vspace{0.8em}

\textbf{Step 1: Uniqueness}. Assume that there exist two solutions $(\mathcal{Y}^n,\mathcal{Z}^n)\in \S^\infty_{\F} \times \H^2_{\F}$ and $(\widetilde{\mathcal{Y}^n},\widetilde{\mathcal{Z}^n})\in \S^\infty_{\F} \times \H^2_{\F}$ to BSDE \eqref{bsde:approx} such that $\|\mathcal{Z}^n\|_{\H^2_{\rm BMO}(\F)}+\|\widetilde{\mathcal{Z}^n}\|_{\H^2_{\rm BMO}(\F)}$ is uniformly bounded in $n$. Set $\delta \mathcal{Y}^n:=\mathcal{Y}^n-\widetilde{\mathcal{Y}^n}$ and $\delta \mathcal{Z}^n:=\mathcal{Z}^n-\widetilde{\mathcal{Z}^n}$, and
$$\widecheck{\lambda^n_s}:=\lambda^n_s e^{\alpha (\xi_s^a-\widetilde{\mathcal Y^n_s})}\int_0^1 e^{-\alpha \theta (\mathcal Y^n_s-\widetilde  {\mathcal{Y}^n_s})} d\theta.$$ Thus $(\delta \mathcal{Y}^n, \delta \mathcal{Z}^n)$ is solution of
\begin{equation}\label{bsde:deltan} \delta \mathcal{Y}_t^n = 0-\int_t^T \eta^b(s,\z^n_s,\widetilde\z^n_s)\cdot \de \z_s^n+\left(\widecheck{\lambda_s^n}+m^b(s,\y_s^n,\widetilde\y_s^n)\right) \delta \mathcal{Y}_s^n ds-\int_t^T \delta \mathcal{Z}^n_s \cdot dW_s.\end{equation}
Hence, knowing that $\int_0^\cdot \mathcal{Z}^n_s \cdot dW_s$ and $\int_0^\cdot \tilde \z^n_s\cdot  dW_s$ are two BMO$(\F)$ martingales and using Assumption \ref{assumpt:g}, we know that $\eta^b$ is in $\mathbb H^2_{{\rm BMO}(\F)}$ and we can define a probability $\Q$ by
$$\frac{d\Q}{d\P}:=\mathcal{E}\left( -\int_0^T \eta^b(s,\z_s^n,\widetilde\z^n_s)\cdot dW_s\right).$$
Moreover, $W^{\Q}:=W+\int_0^\cdot  \eta^b(s,\z^n_s,\widetilde\z_s^n)ds$ is then a Brownian motion under $\Q$. So BSDE \eqref{bsde:deltan} rewrites as
\begin{equation}\label{bsde:deltanq} \delta \mathcal{Y}_t^n = 0-\int_t^T \left(\widecheck{ \lambda_s^n} + m^b(s,\y_s^n,\widetilde\y_s^n)\right)\delta \mathcal{Y}_s^n ds-\int_t^T \delta \mathcal{Z}^n_s \cdot dW^{\Q}_s.\end{equation}
Set $$\widetilde{ \delta \mathcal{Y}}_t^n:=e^{-\int_0^t \widecheck{\lambda_s^n }+m^b(s,\y_s^n,\widetilde\y_s^n) ds}\delta \mathcal{Y}_t^n,\text{ for all $t\in [0,T]$}.$$ Then $(\widetilde {\delta \mathcal{Y}}^n,\widetilde {\delta \mathcal{Z}}^n)$ satisfies
$$\widetilde {\delta \mathcal{Y}}_t^n = 0-\int_t^T e^{-\int_0^s \widecheck{\lambda_u^n} +m^b(u,\y_u^n,\widetilde\y_u^n) ds}\delta \mathcal{Z}_s^n \cdot dW_s^{\Q}, \  t\in [0,T],$$ which admits $(0,0)$ as unique solution.

\vspace{0.5em}

\textbf{Step 2: Existence}.
We turn now to the existence of a solution of BSDE \eqref{bsde:approx} in $\S^\infty_{\F} \times \H^2_{{\rm BMO}(\F)}$. Consider the following truncated BSDE
\begin{align}\label{bsdetrunc}
\widehat{Y}_t^n=\xi^b-\int_{t}^{T}  g^b(s,\widehat Y_s^n,\widehat Z_s^n) +\lambda_s^n  \frac{1-e^{\alpha (\xi^a_s-\widehat Y_s^n\vee(-C_Y))}}{\alpha}ds-\int_{t}^{T}\widehat{Z}_s^n\cdot dW_s.
\end{align}
Then, the classical quadratic BSDE \eqref{bsdetrunc} admits a unique solution $(\widehat{Y}^n,\widehat{Z}^n) \in \S^\infty_{\F} \times \H^2_{{\rm BMO}(\F)}$ (see e.g. \cite{Kobylanski}).  We can then rewrite BSDE \eqref{bsdetrunc} as
\begin{align}\label{bsdetrunclinear}
\widehat{Y}_t^n\nonumber=&\ \xi^b-\int_{t}^{T}\Big(  g^b(s,0,0)+\left(\tilde \lambda_s^n \mathbf{1}_{\widehat{Y}_s^n\geq -C_Y} +m^b(s,\widehat{Y}_s^n)\right) \widehat{Y}_s^n-\tilde\lambda_s^n\xi^a_s\mathbf{1}_{\widehat{Y}_s^n\geq -C_Y}\\
&+\lambda_s^n\frac{1-e^{\alpha (\xi^a_s+C_Y)}}{\alpha}\mathbf{1}_{\widehat{Y}_s^n< -C_Y}+\eta^b(s,\widehat{Z}_s^n)\cdot\widehat{Z}_s^n\Big)ds-\int_{t}^{T}\widehat{Z}_s^n \cdot dW_s,
\end{align}
where $\tilde \lambda_s^n:=\lambda_s\wedge n \int_0^1 e^{-\alpha\theta (\widehat{Y}^n_s-\xi^a_s)} d\theta $.

\vspace{0.5em}
\noindent Set $\gamma^n(s):= \tilde\lambda_s^n \mathbf{1}_{|\widehat{Y}_s^n|\leq C_Y}+m^b(s,\widehat{Y}_s^n)$ and $\mathcal{Y}^n:=\widehat{Y}^ne^{-\int_0^\cdot \gamma^n(s)ds}$, we obtain from BSDE \eqref{bsdetrunclinear}
\begin{align*}
 \mathcal{Y}^n_t=&\ \xi^be^{-\int_0^T \gamma^n_udu}-\int_t^T e^{-\int_0^s \gamma^n_udu}\Big(g^b(s,0,0)-\tilde \lambda_s^n \xi^a_s\mathbf{1}_{\widehat{Y}_s^n\geq -C_Y}\Big)ds\\
&-\int_t^T e^{-\int_0^s\gamma^n_udu}\lambda_s^n\frac{1-e^{\alpha (\xi^a_s+C_Y)}}{\alpha}\mathbf{1}_{\widehat{Y}_s^n< -C_Y}ds-\int_{t}^{T}e^{-\int_0^s\gamma^n_udu}\widehat{Z}_s^n \cdot dW^{\Q^n}_s, ~t\in[0, T],
 \end{align*}
 where $d\Q^n=\mathcal{E}(-\int_0^T\eta^b(s,\widehat{Z}_s^n)\cdot  dW_s)d\mathbb{P}$ and $W^{\mathbb{Q}^n}:=W+\int_0^\cdot \eta^b(s,\widehat{Z}_s^n)ds$ is a Brownian motion under the probability $\Q^n$, since $\int_0^\cdot\eta^b(s,\widehat{Z}_s^n) \cdot dW_s$ is a BMO$(\F)$-martingale from Assumption \ref{assumpt:g}$(ii)$.
Increasing the constants if necessary, we have $\xi^a \geq -C_Y$, then taking the conditional expectation under $\Q^n$ we deduce that
\begin{equation*}\label{ineg:yn}
 \widehat{Y}_t^n\geq -e^{M T}\Big(\| \xi^b\|_\infty+M (T-t)+\| \xi^a\|_\infty \mathbb{E}^{\Q^n}\Big[\underbrace{\int_t^Te^{-\int_t^s \tilde\lambda_u^n \mathbf{1}_{\widehat{Y}_u^n\geq -C_Y}du} \tilde \lambda_s^n  \mathbf{1}_{\widehat{Y}_s^n\geq -C_Y}ds}_{:=I} \Big| \mathcal{F}_t \Big]\Big).\end{equation*}
 Since $I=1-e^{-\int_t^T \tilde\lambda_u^n \mathbf{1}_{\widehat{Y}_u^n\geq -C_Y} du}\leq 1$, we deduce that $ \widehat{Y}_t^n\geq -C_Y$. A posteriori, we deduce that the solution $(\widehat{Y}^n,\widehat{Z}^n)$ of BSDE \eqref{bsdetrunc} is in fact the unique solution $(Y^{b,n},Z^{b,n})$ of BSDE \eqref{bsde:approx} in $\S^\infty_\F\times \H^2_{\text{BMO}(\F)}$ such that $Y_t^{b,n}\geq -C_Y, \; t\in [0,T], \; \P-a.s$. Then, using a linearization and taking the conditional expectation under $\Q^n$, we can compute explicitly $Y^{b,n}$ from BSDE \eqref{bsde:approxbis}
\begin{align*}
 Y_t^{b,n}&=-\E^{\Q^n}\left[\left.\xi^b e^{-\int_t^T \gamma^n_u du}+  \int_t^T e^{-\int_t^s \gamma^n_u du}\Big(g^b(s,0,0)-\tilde \lambda_s^n \xi^a\Big)ds\right| \mathcal{F}_t\right]\\
 &\leq e^{M T}( \| \xi^b\|_\infty+M(T-t)+\| \xi^a\|_\infty).
\end{align*}

\vspace{0.8em}
\textbf{Step 3: BMO norm of $Z^{b,n}$.} Let $\rho \in \mathcal{T}(\F)$ be a random horizon and $\beta$ a positive constant. Using It\^o's formula, we obtain
\begin{align*}
 e^{\beta Y^{b,n}_\rho}=&\ e^{\beta \xi^b}-\int_\rho^T \beta e^{\beta Y_s^{b,n}} \left(  g^b(s,Y_s^{b,n},Z_s^{b,n}) +\lambda_s^n  \frac{1-e^{\alpha (\xi^a_s-Y_s^{b,n})}}{\alpha}\right)ds\\
 & -\int_\rho^T \beta e^{\beta Y_s^{b,n}} Z_s^{b,n} \cdot dW_s -\frac{\beta^2}{2}\int_\rho^T e^{\beta Y_s^{b,n}} \|Z_s^{b,n}\|^2 ds.
 \end{align*} Hence, from Assumption $($H1$)$, using the fact that, by Step 2, $Y^{b,n}$ is uniformly bounded in $n$ by $C_Y$ and taking conditional expectations, we deduce
\begin{align*}
\frac{\beta^2}{2}\E\left[\left.\int_\rho^T e^{\beta Y_s^{b,n}} \|Z_s^n\|^2 ds\right| \mathcal{F}_\rho\right] \leq &\ e^{\beta \| \xi^b\|_\infty}+\beta\mathbb{E}\left[ \left.\int_\rho^Te^{\beta Y_s^{b,n}}|g^b(s,Y_s^{b,n},Z_s^{b,n})|ds\right|\mathcal F_\rho\right]\\[0.5em]
&+\beta e^{\beta C_Y}\frac{1+e^{\alpha(\| \xi^a\|_\infty+C_Y)}}{\alpha}\E\left[\left.\int_\rho^T \lambda _s ds \right| \mathcal{F}_{\rho}\right] .\end{align*} Since $|g^b(s,y,z)|\leq \mu(1+ |y| + \|z\|^2)$ we obtain
\begin{align*}
\left(\frac{\beta^2}{2}-\mu\beta\right)\E\left[\left.\int_\rho^T e^{\beta Y_s^n} \|Z_s^{b,n}\|^2 ds\right| \mathcal{F}_\rho\right] \leq&\ e^{\beta \| \xi^b\|_\infty}+\beta  e^{\beta C_Y}T\mu(1+ C_Y)\\
&+\beta e^{\beta C_Y}\frac{1+e^{\alpha(\| \xi^a\|_\infty+C_Y)}}{\alpha}\E\left[\left.\int_\rho^T \lambda _s ds \right| \mathcal{F}_{\rho}\right].\end{align*}
By choosing $\beta > 2\mu$, under Assumption (H2) and using the boundedness of $Y^{b,n}$, we deduce that
$$ \E\left[\left.\int_\rho^T  \Norm{Z_s^{b,n}}^2 ds\right| \mathcal{F}_\rho\right]  \leq C_\beta,$$ where
$$C_\beta:= e^{2\beta C_Y}\left[1+\beta \left(\frac{1+e^{\alpha(\| \xi^a\|_\infty+C_Y)}}{\alpha}\E\left[\left.\int_\rho^T \lambda _s ds \right| \mathcal{F}_{\rho} \right]+ T\mu(1+ C_Y)\right)\right]\times \frac{1}{\frac{\beta^2}{2}-\mu\beta} .$$ Then, under Assumption $($H2$)$, $\| Z^{b,n}\|_{\H^2_{\rm BMO}(\F)}$ is uniformly bounded in $n$.
\end{proof}

\begin{theorem}
\label{th:mainH2}
Let Assumptions $($H1$)$-$($H2$)$ and Assumption \ref{assumpt:g} hold. Then the Brownian BSDE
\begin{equation}
\label{eq:Broexpo1bis}
Y_t^{b} = \xi^b-\int_{t}^{T}g^b(s, Y_s^{b}, Z_s^{b})+\lambda_s\frac{1-e^{\alpha(\xi^a_s-Y_s^{b})}}{\alpha}ds-\int_{t}^{T}Z_s^{b}\cdot dW_s, ~t\in[0, T],
\end{equation}
admits a unique solution $($in $\S^2_\F\times \H^2_\F)$. In addition, $Y^b$ is bounded and $\int_0^\cdot Z_s^{b} \cdot dW_s$ is a {\rm BMO}$(\F)$-martingale.
\end{theorem}

\begin{proof}
The proof is based on an approximation procedure using BSDE \eqref{bsde:approxbis}. The aim of this proof is to show that the solution $(Y^n,Z^n)$ to this approached BSDE converges in $\S^{\infty}_{\F}\times\H^2_{\text{BMO}(\F)}$ to the solution of BSDE \eqref{eq:Broexpo1bis}. Let $p,q\geq n$, we denote $\delta Y_t:= Y^p_t-Y^q_t$ and $\delta Z_t:= Z_t^p-Z_t^q$ for all $t\in [0,T]$. Then, $(\delta Y, \delta Z)$ is solution of the following BSDE
\begin{align*}
\delta Y_t=&-\int_t^T m^b(s,Y_s^p,Y_s^q)\delta Y_s +\eta^b(s,Z_s^p,Z_s^q)\cdot \de Z_s+\lambda_s^p \frac{1-e^{\alpha (\xi^a_s-Y_s^p)}}{\alpha}ds\\
&-\int_t^T \lambda_s^q \frac{1-e^{\alpha (\xi^a_s-Y_s^q)}}{\alpha} ds-\int_t^T \de Z_s\cdot dW_s,
 \end{align*}
which can be rewritten as
\begin{align*}
\delta Y_t=&- \int_t^T\underbrace{\frac{ \lambda_s^p-\lambda_s^q}{\alpha}+ (\lambda_s^p-\lambda_s^q)\frac{e^{\alpha (\xi^a_s-Y_s^p)}}{\alpha}}_{:=\varphi_s^{p,q}}+\left(\lambda_s^q e^{\alpha( \xi^a_s-\overline{Y_s})}+m^b(s,Y_s^p,Y_s^q)\right)\de Y_sds\\
& -\int_t^T \delta Z_s\cdot  dW^{\Q^n}_s,
\end{align*}
where $\overline{Y}$ is a process lying between $Y^p$ and $Y^q$ which satisfies for all $s\in [t,T]$, $|\overline{Y}_s|\leq C_Y, \; \P-a.s.$, and where $W^{\Q^n}:=W+\int_0^\cdot \eta^b(s,Z_s^p,Z_s^q) ds$ is a Brownian motion under $\Q^n$ given by
$$ \frac{d\Q^n}{d\P}= \mathcal{E}\left( -\int_0^T \eta^b(t,Z_t^p,Z_t^q) \cdot dW_t\right),$$ which is well defined since $\int_0^\cdot \eta^b(s,Z_s^p,Z_s^q)\cdot dW_s$ is a BMO$(\F)$ martingale from Assumption \ref{assumpt:g}. Let $\beta \geq 0$, using It\^o's formula
\begin{align*}
e^{\beta t}|\delta Y_t|^2=&\ 0-\int_t^T 2e^{\beta s}\delta Y_s \varphi_s^{p,q}+e^{\beta s}\left(2\lambda_s^q e^{\alpha (\xi^a_s- \overline{Y_s})}+2m^b(s,Y_s^p,Y_s^q)+\beta\right)|\delta Y_s|^2 ds\\
&-2\int_t^T e^{\beta s} \delta Y_s \delta Z_s \cdot dW_s^{\Q^n}-\int_t^T e^{\beta s}\|\delta Z_s\|^2 ds.
\end{align*} Using the non-negativity of $\lambda^q$ and choosing $\beta>2M$, we deduce that
$$e^{\beta t} |\delta Y_t|^2\leq 0-\int_t^T 2e^{\beta s}\delta Y_s \varphi_s^{p,q}ds-2\int_t^T e^{\beta s}\delta Y_s \delta Z_s \cdot dW_s^{\Q^n}-\int_t^T e^{\beta s}\|\delta Z_s\|^2 ds.$$ Then, using the boundedness of $Y^n$ uniformly in $n$, there exists a positive constant $C$ such that
$$ \E^{\Q^n}\left[ \sup_{t\in [0,T]} |\delta Y_t|^2\right] +\E^{\Q^n}\left[ \int_0^T \|\delta Z_s\|^2 ds\right] \leq C\E^{\Q^n}\left[ \int_0^T |\lambda_s^p-\lambda_s^q| ds\right],$$
Hence,
\begin{equation}\label{inegcauchyqn} \E^{\Q^n}\left[ \sup_{t\in [0,T]} |\delta Y_t|^2\right]  \leq C\E^{\Q^n}\left[ \int_0^T |\lambda_s^p-\lambda_s^q| ds\right].\end{equation}  We want to obtain this kind of estimates under the probability $\P$. Notice that
\begin{align*}
\E\left[ \sup_{t\in [0,T]} |\delta Y_t|^2\right]&=\E^{\Q^n}\left[ \mathcal{E}\left(-\int_0^T \eta^b(t,Z_t^p,Z_t^q) \cdot dW_t\right)^{-1}\sup_{t\in [0,T]} |\delta Y_t|^2\right]\\
&=\E^{\Q^n}\left[ \mathcal{E}\left(\int_0^T \eta^b(t,Z_t^p,Z_t^q)\cdot dW^{\Q^n}_t\right)\sup_{t\in [0,T]} |\delta Y_t|^2\right].
\end{align*}
From Assumption \ref{assumpt:g} and Lemma \ref{lem:aprioriestimates}, $\int_0^\cdot \eta^b(s,Z_s^n)\cdot dW_s$ is a BMO($\F$) martingale and $\|\eta^b(\cdot,Z_\cdot^n)\|_{\H^2_{\rm BMO}(\F)}$  is uniformly bounded in $n$. Then according to \cite[Theorem 3.3]{kazamaki}, $\int_0^\cdot \eta^b(s,Z_s^n)\cdot dW_s^{\Q^n}$ is a BMO($\Q^n,\F$) martingale. Moreover, following the proof of \cite[Theorem 3.3]{kazamaki} together with the proof of \cite[Theorem 2.4]{kazamaki}, it is easily verified that $\| \eta^b(\cdot,Z_\cdot^n)\|_{\H^2_{\rm BMO}(\Q^n,\F)}$ is uniformly bounded in $n$. Thus, from \cite[Theorem 3.1]{kazamaki} there exists $r>1$ (its conjugate being denoted by $\overline{r}$) such that $$\underset{n\geq 1}{\sup}\ \E^{\Q^n}\left[  \mathcal{E}\left(\int_0^T \eta^b(t,Z_t^p,Z_t^q)\cdot dW^{\Q^n}_t\right)^r\right]<+\infty.$$
Since $Y^n$ is uniformly bounded in $n$, we deduce that there exists $k>0$ such that
\begin{align}
\E\left[ \sup_{t\in [0,T]} |\delta Y_t|^2\right]  &\nonumber\leq \E^{\Q^n}\left[ \mathcal{E}\left(\int_0^T \eta^b(t,Z_t^p,Z_t^q) \cdot dW^{\Q^n}_t\right)^r \right]^{\frac1r}\E^{\Q^n}\left[ \sup_{t\in [0,T]} |\delta Y_t|^{2\overline{r}}\right]^{\frac{1}{\overline{r}}}\\
\label{leftside}&\leq k\E^{\Q^n}\left[ \sup_{t\in [0,T]} |\delta Y_t|^{2}\right]^{\frac{1}{\overline{r}}}.
\end{align}
Similarly, from the definition of $\Q^n$ there exists $K>0$ such that
\begin{equation}\label{rightside}
\E^{\Q^n}\left[ \int_0^T |\lambda_s^p-\lambda_s^q| ds\right]\leq K\E\left[ \left( \int_0^T |\lambda_s^p-\lambda_s^q| ds\right)^{\overline{r}}\right]^{\frac{1}{\overline{r}}}.
\end{equation}
Thus, from Inequalities \eqref{inegcauchyqn}, \eqref{leftside} and \eqref{rightside}, we deduce that there exists a positive constant $\kappa$ such that Inequality \eqref{inegcauchyqn} rewrites
$$\E\left[ \sup_{t\in [0,T]} |\delta Y_t|^2\right]\leq \kappa\E\left[ \left( \int_0^T |\lambda_s^p-\lambda_s^q| ds\right)^{\overline{r}}\right]^{\frac{2}{\overline{r}}}\underset{n\to \infty}{\longrightarrow} 0,$$
by dominated convergence and using \eqref{energy1}.

\vspace{0.5em}
\noindent Then, we deduce that $Y^n$ is a Cauchy sequence in $\S^2_\F$. Hence, $Y^n$ converges in $\S^2_\F$ to a process $Y$. Besides, since $Y^{b,n}$ is uniformly bounded in $n$, taking a subsequence (which we still denote $(Y^{{b,n}})_{n\geq 0}$ for simplicity), of uniformly bounded process in $n$ which converges, $\P-a.s.$, to $Y^b$, we deduce that $Y^b\in \mathbb{S}^\infty_\F$. Thus, by Lebesgue's dominated convergence Theorem, $Y^{b,n}$ converges to $Y^b$ in $\S^p_\F$ for every $p \geq 1$. Recall that
$$Y_t^{b,n}=\xi^b-\int_{t}^{T} g^b(s,Y_s^{b,n},Z_s^{b,n})+\tilde\lambda_s^n Y_s^{b,n}-\tilde\lambda_s^n \xi^a_s ds-\int_{t}^{T}Z_s^{b,n} \cdot dW_s, $$ where $\tilde\lambda_s^n:= \lambda_s^n\int_0^1 e^{-\alpha\theta (Y_s^{b,n}-\xi_s^a)}d\theta$, which can be rewritten

$$ Y_t^{b,n}=Y_0^{b,n}+\int_{0}^{t}   A^{b,n}_s  ds+\int_{0}^{t}Z_s^{b,n}\cdot dW_s, $$ where
$A_s^n:=g^b(s,Y_s^{b,n},Z_s^{b,n})+\tilde\lambda_s^n Y_s^{b,n}-\tilde\lambda_s^n \xi^a_s$. Knowing that $\lim_{n\to \infty} \|Y^{b,n}-Y^b\|_{\S^p_\F} =0$ for every $p \geq 1$, we deduce  from Theorem 1 in \cite{BarlowProtter} that $Y^b$ is a semimartingale such that $Y^b_t=Y^b_0+\int_0^t A_s ds+\int_0^t Z^b_s\cdot dW_s$, where for all $p\geq 1$

$$ \E\left[\left(\sup_{t\in [0,T]}\int_0^t Z^b_s\cdot dW_s\right)^p \right]\leq K, \  \E\left[\left( \int_0^T |A_s| ds\right)^p\right]\leq K,$$  for some positive constant $K$, and
$$\lim_{n\to \infty} \E\left[\left( \int_0^T |Z^{b,n}_s-Z^b_s|^2 ds\right)^\frac p2\right]=0, \  \lim_{n\to \infty} \E\left[ \left(\int_0^T |A_s^n-A_s| ds\right)^p\right]=0.$$
Since $\int_0^\cdot Z^{b,n}_s\cdot dW_s$ is a BMO($\F$) martingale, there exists $K'>0$ such that $\|Z^{b,n}\|_{\H^p_\F}+\|Z^b\|_{\H^p_\F} \leq K'$. Besides, using the fact that $Y^{b,n},Y^b\in \S^\infty$, there exists a positive constant $C$ which may vary from line to line such that
\begin{align*}
&\E\left[\left(\int_0^t |A^n_s -\left(g^b(s,Y^b_s,Z^b_s) +\tilde\lambda_s Y^b_s- \tilde\lambda_s\xi_s^a\right)|ds\right)^p \right] \\
&\leq C\Bigg(\E\left[\left(\int_0^t  | Y^b_s-Y_s^{b,n}|ds\right)^p\right] +\E\left[ \left(\int_0^t (1+\| Z^b_s\|+\|Z_s^{b,n}\|) \| Z^b_s-Z_s^{b,n}\|ds\right)^p\right]\\
&\hspace{0.9em}+\E\left[\left(\int_0^t  \left| \tilde\lambda_s Y^b_s -\tilde\lambda_s^n Y_s^{b,n} \right| ds\right)^p\right]+\E\left[\left(\int_0^t  \left| \tilde\lambda_s -\tilde\lambda_s^n \right| |\xi^a_s|ds\right)^p\right]\Bigg)\\
&\leq C\Bigg( \| Y^b-Y^{b,n}\|_{\S^p}+  \E\left[\left( \int_0^T \|Z^{b,n}_s-Z^b_s\|^2 ds\right)^p\right]^\frac12\\
&\hspace{0.9em}+ \E\left[\left( \int_0^t |\lambda_s-\lambda_s^n|ds\right)^p\right]+\| Y^{b,n}-Y^b\|_{\S^p_\F}\E[\Lambda_t^p]\Bigg) \\
&\underset{n\to \infty}{\longrightarrow}0.
\end{align*}
Then, we deduce that there exists a $\mathbb{F}$-predictable process $Z^b$ such that
$$ Y^b_t=Y^b_0 +\int_0^tg^b(s,Y^b_s,Z^b_s) +\tilde\lambda_s Y^b_s-\tilde\lambda_s \xi_s^a ds +\int_0^t Z^b_s\cdot dW_s.$$ Following the Step 3 in the proof of Lemma \ref{lem:aprioriestimates}, we deduce that $Z^b\in \H^2_{\text{BMO}(\F)}$ Then, the pair $(Y^b,Z^b) \in \S^\infty \times \H^2_{\text{BMO}(\F)}$ built previously is the unique solution of BSDE \eqref{eq:Broexpo1bis}, the uniqueness coming from Lemma \ref{lemme:expounique} together with Proposition \ref{propBroH2}.
\end{proof}
\noindent We now turn to Assumption $($H2'$)$. Notice that the proof of Theorem \ref{th:mainH2} fails under $($H2'$)$ since $\E\left[ \Lambda_T\right]=\infty$. We need more regularity on $\xi^a$ to get a sign on $Y^{b,n}$, the first component of the solution of the approached BSDE \eqref{bsde:approx} in order to prove that BSDE \eqref{edsr:brogb} admits a solution under $($H2'$)$.

\begin{assumption}\label{assumpt:xia}
$\xi^a$ is a bounded semi-martingale such that $$\xi^a_t=\xi^a_0+\int_0^t D_s ds+\int_0^t \gamma_s \cdot dW_s,$$ where $D,\gamma$ are bounded processes satisfying for all $s\in [0,T]$, $ g^b(s,\xi^a_s,\gamma_s)- D_s \geq 0$.
\end{assumption}

\noindent Before going further, to solve the utility maximization problem \eqref{eq:pb} according to Theorem \ref{th:expo}, we have to prove that $\int_0^\cdot Z_s dW_s+\int_0^\cdot (e^{\alpha U_s}-1) dM_s$ is a BMO$(\G)$-martingale. Under Assumption $($H2$)$, this property comes for free from the BMO$(\F)$-martingale property of $\int_0^\cdot Z_s^b dW_s$ and the boundedness of $Y^b$. However, under $($H2'$)$ it is not clear that whether the BMO$(\F)$-martingale property implies the BMO$(\G)$-martingale property. It is why we show that under $($H2'$)$, BSDE \eqref{edsr:brogb} admits a unique solution in $\S^\infty_\F\times \H_{{\rm{BMO}}(\G)}^2$, as a consequence of the Immersion hypothesis, which is itself a consequence of $($H1$)$.
\begin{lemma}\label{lem:edsrdegeneree} Assume that $($H1$)$-$($H2'$)$  and Assumptions \ref{assumpt:g} and \ref{assumpt:xia} hold. Then, the following BSDE
\begin{equation}\label{edsr:nonzero} Y^b_t=A-\int_t^Tg^b(s,Y^b_s+\xi_s^a ,Z^b_s+\gamma_s)-D_s +\lambda_s f(Y^b_s)  ds-\int_t^T Z^b_s\cdot  dW_s,\end{equation}where $f(x):=\frac{1-e^{-\alpha x}}{\alpha}$ admits a solution in $\S^\infty_\F\times \H_{{\rm{BMO}}(\G)}^2$ if and only of $A\equiv0$. In this case, the solution is unique.
\end{lemma}

\begin{proof}
Assume that $A\equiv 0$. We aim at showing that BSDE \eqref{edsr:nonzero} admits a (unique) solution in $\S^\infty_\F \times \H^2_{\text{BMO}(\G)}$. Consider the truncated BSDE
\begin{equation}\label{approachedbsde:phi}Y^{b,n}_t=0-\int_t^T g^b(s,Y_s^{b,n}+\xi_s^a,Z_s^{b,n}+\gamma_s)-D_s +\lambda^n_s f(Y^{b,n}_s) ds-\int_t^T Z^{b,n}_s \cdot dW_s,\end{equation}which can be rewritten under Assumption \ref{assumpt:g}
$$  Y^{b,n}_t=0-\int_t^Tg^b(s,\xi_s^a,\gamma_s)-D_s +m_s Y_s^{b,n}+\eta_s\cdot Z_s^{b,n} +\tilde\lambda^n_s Y_s^{b,n}  ds-\int_t^T Z^{b,n}_s \cdot dW_s,$$ with $m_s:= m(s,Y_s^{b,n}+\xi_s^a, \xi_s^a)$, $\eta_s:= \eta(s,Z_s^{b,n}+\gamma_s, \gamma_s)$ and  $\tilde\lambda^n_s:=\lambda_s^n\int_0^1 e^{-\alpha \theta Y_s^{b,n}}d\theta$. Then, following Step 1 and Step 2 in the proof of Lemma \ref{lem:aprioriestimates} and since $g^b(s,\xi_s^a, \gamma_s)-D_s$ is non-negative under Assumption \ref{assumpt:xia}, we show that BSDE \eqref{approachedbsde:phi} admits a unique solution $(Y^{b,n},Z^{b,n}) \in \S^\infty_\F \times \H^2_{\text{BMO}(\F)}$ such that
\begin{equation}\label{eq:estim}
-e^{M T}(T-t)\tilde M\leq Y_t^{b,n} \leq 0, \text{ for all $t\in [0,T], \; \P-a.s.,$}
\end{equation}  where $\tilde M$ is  a positive constant. We show now that the $\H^2_{\text{BMO}(\G)}$ norm of $Z^{b,n}$ does not depend on $n$ by following Step 3 of the proof of Lemma \ref{lem:aprioriestimates}. Let $\rho \in \mathcal{T}(\G)$ be a random horizon and $\beta<0$. Using It\^o's formula, we obtain
\begin{align*}
 e^{\beta Y^{b,n}_\rho}=&\ 1-\int_\rho^T \beta e^{\beta Y_s^{b,n}} \left(  g^b(s,Y_s^{b,n}+\xi_s ^a,Z_s^{b,n}+\gamma_s)-D_s +\lambda_s^n  \frac{1-e^{-\alpha Y_s^{b,n}}}{\alpha}\right)ds\\
 &-\int_\rho^T \beta e^{\beta Y_s^{b,n}} Z_s^{b,n} \cdot dW_s -\frac{\beta^2}{2}\int_\rho^T e^{\beta Y_s^{b,n}} \|Z_s^{b,n}\|^2 ds.
 \end{align*} Hence, using the fact that $Y^{b,n}$ is non positive and uniformly bounded in $n$ and taking conditional expectations, we have for any $\rho\in\mathcal T(\G)$, from the Immersion property $($H1$)$
\begin{align*}
\frac{|\beta|^2}{2}\E\left[\left.\int_\rho^T e^{\beta Y_s^{b,n}} \|Z_s^{b,n}\|^2 ds\right| \mathcal{G}_\rho\right]& \leq 1 +|\beta| \mathbb{E}\left[\left.\int_\rho^Te^{\beta Y_s^{b,n}} |D_s| ds\right|\mathcal G_\rho\right]\\
&+|\beta|\mathbb{E}\left[\left.\int_\rho^Te^{\beta Y_s^{b,n}}|g^b(s,Y_s^{b,n}+\xi_s ^a,Z_s^{b,n}+\gamma_s)|ds\right|\mathcal G_\rho\right].
\end{align*}
Since $\xi^a$, $D$ and $\gamma$ are bounded, using the fact that $|g^b(s,y,z)|\leq \mu(1+ |y| + \|z\|^2)$, we obtain
\begin{align*}
\left(\frac{|\beta|^2}{2}-2\mu|\beta|\right)\E\left[\left.\int_\rho^T e^{\beta Y_s^{b,n}} \|Z_s^{b,n}\|^2 ds\right| \mathcal{G}_\rho\right] &\leq e^{|\beta | \| \xi^b\|_\infty}+|\beta |  e^{\beta \| Y^{b,n}\|_{\infty}}C(1+ \| Y^{b,n}\|_{\infty}),
\end{align*} with $C>0$. Choosing $\beta >4\mu$ and using the boundedness of $Y^{b,n}$ uniformly in $n$, we deduce that there exists a constant $C>0$ which does not depend on $n$ such that
$$ \E\left[\left.\int_\rho^T  \|Z_s^{b,n}\|^2 ds\right| \mathcal{G}_\rho\right]  \leq C.$$
Thus, $\| Z^{b,n}\|_{\H^2_{\text{BMO}(\G)}}$ is uniformly bounded in $n$.

\vspace{0.5em}
\noindent We prove now the convergence of the sequence $(Y^{b,n})$ in $\S^p_\F$ for every $p$ in order to apply Theorem 1 of \cite{BarlowProtter}. Recall that $Y^{b,n}_t\leq 0$, for every $t\in [0,T]$. Then, from the comparison theorem for quadratic BSDEs (see e.g. \cite[Theorem 2.6]{Kobylanski}) and since $Y^{b,n}$ is non positive, the sequence $(Y^{b,n})_n$ is non-decreasing. Hence, it converges almost surely to $$Y^b_t:=\lim\limits_{n\to \infty} Y_t^{b,n} \text{ such that $-e^{M T}(T-t) M\leq Y^b_t\leq 0$ for all $t\in [0,T]$.}$$  Fix $0<t_0<T$, we notice that $(Y^{b,n},Z^{b,n})$ is also the solution to the following BSDE for $0\leq t \leq t_0$
$$ Y^{b,n}_t=Y^{b,n}_{t_0}+\int_t^{t_0} g^b(s,Y^{b,n}_s+\xi_s^a,Z_s^{b,n}+\gamma_s)-D_s +\lambda^n_s f(Y^{b,n}_s) ds -\int_t^{t_0} Z_s^{b,n}\cdot  dW_s.$$
Hence, for every $n\geq 1$ and $p,q\geq n$, by setting $\delta Y:= Y^{b,p}-Y^{b,q}$ and reproducing the proof of Theorem \ref{th:mainH2} with $t_0<T$ as terminal time instead of $T$, we deduce that for every $r\geq 0$ there exists $C_r>0$ which does not depend on $p,q$ such that
$$ \E\left[ \sup_{t\in [0,t_0]} | \delta Y_t|^2\right] \leq C_r\left(\E\left[ | \delta Y_{t_0}|^2\right] + \E\left[ \left( \int_0^{t_0} |\lambda_s^p-\lambda_s^q| ds\right)^{\overline{r}}\right]^{\frac{2}{\overline{r}}}\right).$$
Hence, there exists $\tilde{C}>0$ such that for every $n\geq 0$
$$ \sup_{p,q\geq n}\E\left[ \sup_{t\in [0,t_0]} | \delta Y_t|^2\right] \leq \tilde{C}\left(\E\left[ | Y_{t_0}^{b,n}-Y^b_{t_0}|^2\right] + \E\left[ \left( \int_0^{t_0} |\lambda_s^n-\lambda_s| ds\right)^{\overline{r}}\right]^{\frac{2}{\overline{r}}}\right).$$
By Lebesgue's dominated convergence Theorem and since $\E\left[ \Lambda^{\overline{r}}_{t_0}\right] <\infty$, we deduce that the sequence $(Y^{b,n} \mathbf{1}_{[0,t_0]})$ is a Cauchy sequence in $\S^2_\F$, and knowing that $Y^n$ is uniformly bounded in $n$, $(Y^{b,n} \mathbf{1}_{[0,t_0]})$ is a Cauchy sequence in $\S^p_\F$ for every $p\geq 1$. Thus, $Y^{b,n} \mathbf{1}_{[0,t_0]}$ converges to $Y^b \mathbf{1}_{[0,t_0]}$ in $\S^p_\F$ for every $p\geq 1$. As in the proof of Theorem \ref{th:mainH2}, we deduce  from Theorem 1 in \cite{BarlowProtter} that $Y^b$ is a semimartingale such that for every $t<T$,
$$Y^b_t=Y^b_0+\int_0^t A_s ds+\int_0^t Z^b_s\cdot dW_s,$$
where for all $p\geq 1$ and $0\leq t_0<T$
$$ \E\left[\left(\sup_{t\in [0,t_0]}\int_0^t Z^b_s\cdot dW_s\right)^p \right]\leq K, \ \  \E\left[\left( \int_0^{t_0} |A_s| ds\right)^p\right]\leq K,$$
for some $K>0$, and
$$\lim_{n\to \infty} \E\left[\left( \int_0^{t_0} |Z^{b,n}_s-Z^b_s|^2 ds\right)^\frac p2\right]=0, \  \lim_{n\to \infty} \E\left[ \left(\int_0^{t_0} |A_s^n-A_s| ds\right)^p\right]=0.$$
Hence, there exists a $\F$-predictable process $Z^b$ such that for every $0\leq t<T$
$$ Y^b_t=Y^b_0 +\int_0^t g^b(s,Y^b_s+\xi_s^a,Z^b_s+\gamma_s)-D_s+\tilde\lambda_s Y^b_s ds +\int_0^t Z^b_s \cdot dW_s.$$
Thus, for $\varepsilon>0$ we deduce that there exists a $\F$-predictable process $Z^b$ such that for every $0\leq t<T$
\begin{equation}\label{edsr:epsilon}Y^b_t=Y^b_{(T-\varepsilon) \vee t}-\int_t^{(T-\varepsilon) \vee t}g^b(s,Y^b_s+\xi_s^a,Z^b_s+\gamma_s)-D_s+ \tilde\lambda_s Y^b_sds +\int_t^{(T-\varepsilon) \vee t} Z^b_s\cdot dW_s.
\end{equation}
Moreover using \eqref{eq:estim}
$$ | Y^b_{(T-\varepsilon)\vee t}|=\lim_{n\to \infty} | Y^{b,n}_{(T-\varepsilon)\vee t}| \leq \varepsilon M e^{\mu T} \underset{\varepsilon \to 0}{\longrightarrow} 0=Y^b_T,$$ which implies $Y^b_t$ is continuous at $t=T$. Then, taking the limit when $\varepsilon$ goes to 0 in \eqref{edsr:epsilon}, the pair of processes $(Y^b,Z^b)$ satisfies BSDE \eqref{edsr:nonzero}. Besides, we have proved that $Y^b$ is in $\S^\infty_\F$ and non positive. Hence, following the same lines of the proof of the uniform boundedness of $\|Z^{b,n}\|_{\H^2_{\text{BMO}(\G)}}$, we deduce that $\| Z^b \|_{\H^2_{\rm BMO}(\G)}<+\infty$. Since $Y^b$ is bounded and since $\| Z^b \|_{\H^2_{\rm BMO}(\G)}<+\infty$, we deduce that $(Y^b,Z^b)$ is the unique solution in $\S^\infty_\F\times \H_{{\rm{BMO}}(\G)}^2$ of \eqref{edsr:nonzero}, in the sense of Definition \ref{def:sol:broh3'}.

\vspace{0.5em}
\noindent Assume now that there exists a solution $(Y^b,Z^b)$. Following the Step 1 of the proof of \cite[proposition 3.1]{JR}, we show that necessarily $A \equiv 0$.
\end{proof}

\begin{theorem}
\label{th:mainBH2'}
Assume $($H1$)$-$($H2'$)$ hold. Assume moreover that Assumption \ref{assumpt:xia} holds. Then under Assumption \ref{assumpt:g} the BSDE
\begin{equation}\label{edsr:mainBH2'}Y_t^{b} = \xi^a_T-\int_{t}^{T}f^b(s, Y_s^{b}, Z_s^{b}, \xi^a_s-Y_s^{b})ds-\int_{t}^{T}Z_s^{b}\cdot dW_s, ~t\in[0, T],\end{equation}
with
$$f^b(s,y,z,u):=g^b(s,y,z) + \lambda_s \frac{1-e^{\alpha u}}{\alpha},$$
admits a unique solution such that $Y^b$ is bounded and $\int_0^\cdot Z_s^{b} dW_s$ is a {\rm BMO}$(\G)$-martingale.
\end{theorem}

\begin{proof} Consider the following BSDE
\begin{equation}\label{edsr:tilde} \tilde{Y}^b_t=0-\int_t^Tg^b(s,\tilde Y^b_s + \xi_s^a, \tilde Z^b_s+\gamma_s)-D_s+\lambda_s \tilde f(\tilde{Y}^b_s) ds -\int_t^T \tilde{Z}^b_s \cdot dW_s,\end{equation} where $\tilde f(x):=\frac{1-e^{-\alpha x}}{\alpha}$. Then, according to Lemma \ref{lem:edsrdegeneree}, BSDE \eqref{edsr:tilde} admits a unique solution $(\tilde{Y}^b,\tilde{Z}^b)\in \S^\infty \times \H_{\text{BMO}(\G)}^2$. By setting $Y^b_t:= \tilde{Y}^b_t+\xi^a_t$ and $Z^b_t:=\tilde{Z}^b_t +\gamma_t$, we deduce that $(Y^b,Z^b)$ is the unique solution of \eqref{edsr:mainBH2'} in $\S^\infty_\F \times \H_{\text{BMO}(\G)}^2$.
\end{proof}

\begin{remark}  Even if Assumption \ref{assumpt:xia} is not too restrictive, especially from the point of view of financial application, we would like to point out the fact that it is not a necessary condition. Consider for simplicity the setting corresponding to $\alpha=0$, and assume that $\xi^a$ is a deterministic continuous function of time $($which may be of unbounded variation and thus not a semimartingale$)$, and consider under $($H2'$)$ the following linear BSDE
\begin{equation}\label{linearbsde} Y_t=
\xi^a_T+\int_t^T \lambda_s(\xi_s^a-Y_s) ds- \int_t^T Z_s\cdot dW_s.\end{equation} Assume that it admits a solution. Then, we necessarily have
$$ Y_t=\mathbb{E}\left[\left. \int_t^T \lambda_s e^{-\int_t^s \lambda_u du}\xi_s^a ds\right| \mathcal{F}_t\right].$$ Since $\xi^a$ is automatically uniformly continuous on $[0,T]$, there is some modulus of continuity $\rho$ such that
$$ \left|Y_{T-\varepsilon}- \xi^a_T\right| \leq \rho(\varepsilon)\mathbb{E}\left[\left. \int_{T-\varepsilon}^T \lambda_s e^{-\int_{T-\varepsilon}^s \lambda_u du} ds\right| \mathcal{F}_{T-\varepsilon}\right]=\rho(\varepsilon),$$ so that we obtain $Y_{T-\varepsilon}\longrightarrow \xi^a_T$ when $\varepsilon \longrightarrow 0$.

\vspace{0.5em}
\noindent However, we cannot hope to solve BSDE \eqref{linearbsde} without assuming at least that $\xi^a$ is left-continuous at time $T$. Indeed, assume that $\xi^a=\mathbf{1}_{[0,T)}$ and choose $\lambda_s=\frac{1}{T-s}$. Then,
$$ Y_{T-\varepsilon}=-1 \underset{\varepsilon \to 0}{\nrightarrow} \xi^a_T=0,$$ which means in this case that BSDE \eqref{linearbsde} does not admit a solution.\end{remark}

\noindent The previous remark leads us to hypothesize that Assumption \ref{assumpt:xia} is not necessary to obtain existence and uniqueness of the solution to BSDE \eqref{edsr:mainBH2'}. We give the following conjecture that we leave for future research.

\begin{conjecture}
Assume $($H1$)$-$($H2'$)$ hold and that $g^b(s,0,0)$ is non-negative for every $s\in [0,T]$. Then under Assumption \ref{assumpt:g} the BSDE
\begin{equation*}Y_t^{b} = \xi^a_{T-}-\int_{t}^{T}f^b(s, Y_s^{b}, Z_s^{b}, \xi^a_s-Y_s^{b})ds-\int_{t}^{T}Z_s^{b}\cdot dW_s, ~t\in[0, T],\end{equation*}
with
$$f^b(s,y,z,u):=g^b(s,y,z) + \lambda_s \frac{1-e^{\alpha u}}{\alpha}$$
admits a unique solution such that $Y^b$ is bounded and $\int_0^\cdot Z_s^{b} dW_s$ is a {\rm BMO}$(\G)$-martingale.
\end{conjecture}

\subsection{Existence and uniqueness Theorem for BSDE \eqref{BSDE:Y:General:Case} }
\begin{theorem}
\label{thm:expounique}
Let Assumptions \ref{assumpt:g} and $(H1)$-$(H2)$ be in force. Then under $(H2)$ $($respectively under $(H2')$ and Assumption \ref{assumpt:xia}$)$, BSDE \eqref{BSDE:Y:General:Case} $($recalled below$)$
$$ Y_t = \xi -\int_\t^\T Z_s \cdot dW_s -\int_\t^\T U_s dH_s - \int_\t^\T f(s,Y_s,Z_s,U_s) ds, \  t\in [0,T], $$
admits a unique solution $(Y,Z,U)$ such that $Y$ and $U$ are in $\S^\infty_\G$ and $\int_0^\cdot Z_s\cdot dW_s+\int_0^\cdot (e^{\alpha U_s}-1) dM_s$ is a {\rm BMO}$(\G)$-martingale.
\end{theorem}

\begin{proof}
We have shown the uniqueness of the solution in Lemma \ref{lemme:expounique}. The existence under $($H2$)$ (resp. $($H2'$)$) of a triplet of processes $(Y,Z,U)$ satisfying BSDE \eqref{BSDE:Y:General:Case}, comes directly from Theorem \ref{th:mainH2} (resp. Theorem \ref{th:mainBH2'}) together with Proposition \ref{propBroH2} (resp. Proposition \ref{prop:BroH2'}). We know moreover that $Y$ and $U$ are in $\S^\infty_\G$ and using the Immersion hypothesis, as a consequence of $($H1$)$, $\int_0^\cdot Z_s\cdot dW_s$ is a {\rm BMO}$(\G)$ martingale. Recall that $U_s=(\xi^a_s-Y^b_s)\mathbf{1}_{s\leq \tau}$, where $Y^b$ is the first component of the solution of the Brownian BSDE \eqref{edsr:brogb}. We prove that $\int_0^\cdot (e^{\alpha U_s}-1) dM_s$ is a {\rm BMO}$(\G)$ martingale

\vspace{0.5em}
\noindent \textbf{Under $($H2$)$.} We obtain directly from the definition of $($H2$)$ and since $Y^b,\xi^a$ are bounded
$$ \esssup_{\rho \in \cal T(\G)} \E\left[\left. \int_\rho^T |e^{\alpha (\xi^a_s-Y_t^{b})}-1|^2 \lambda_t dt \right| \cal G_\rho\right]<+\infty.$$

\vspace{0.5em}
\noindent \textbf{Under $($H2'$)$.} We first consider the Brownian BSDE \eqref{edsr:tilde} that we recall

$$  \tilde{Y}^b_t=0-\int_t^Tg^b(s,\tilde Y^b_s + \xi_s^a, \tilde Z^b_s+\gamma_s)-D_s+\lambda_s \tilde f(\tilde{Y}^b_s) ds -\int_t^T \tilde{Z}^b_s \cdot dW_s,$$ where $\tilde f(x):=\frac{1-e^{-\alpha x}}{\alpha}$. Using Decomposition \eqref{decomposition:gb}, we obtain
\begin{align*} \tilde Y_t^{b} =&\ 0-\int_{t}^{T} g^b(t,\xi_s^a,\gamma_s)-D_s+m^b(t,\tilde Y_s^b +\xi_s^a,\xi_s^a)\tilde Y_s^b+\eta^b(t,Z_s^b+\gamma_s,\gamma_s)\cdot \tilde Z_s^bds\\
&-\int_t^T\lambda_s\frac{1-e^{-\alpha \tilde Y_s^{b}}}{\alpha}ds-\int_{t}^{T}\tilde Z_s^{b}\cdot dW_s,\end{align*}
which can be rewritten
\begin{align}
\nonumber \tilde Y_t^{b}=&\ 0-\int_{t}^{T} g^b(t,\xi_s^a,\gamma_s)-D_s+(m^b(t,\tilde Y_s^b +\xi_s^a,\xi_s^a)+\tilde\lambda_s )\tilde Y_s^b+\eta^b(t,\tilde Z_s^b+\gamma_s,\gamma_s)\cdot \tilde Z_s^b ds\\
& \label{bsde:approxbisthm}  -\int_{t}^{T}\tilde Z_s^{b} \cdot dW_s, ~t\in[0, T],\end{align}
where $\tilde\lambda_s:= \lambda_s\int_0^1 e^{-\alpha\theta \tilde Y_s^{b}}d\theta$. Since $m^b$ is bounded by $M>0$, following the proof of \cite[Theorem 4.4]{JR} we can easily show\footnote{Taking $f(x)=\frac{1-e^{-\alpha x}}{\alpha}, \; \delta=1$ in \cite[Theorem 4.4]{JR} and changing $\tilde{\lambda}$ in \cite[Relation (4.4)]{JR} by $\tilde\lambda +m^b$.} that
$$ -e^{MT}e^{\Lambda_t} \E\left[\left. \int_t^T e^{-\Lambda_s} (\varphi_s +|\eta(s,\tilde Z_s^b+\gamma_s,\gamma_s)\cdot \tilde Z_s^b|) ds \right| \cal G_t\right] \leq \tilde Y_t^{b}\leq 0, \  \forall t\in [0,T], $$
where $\varphi_s:=g^b(t,\xi^a_s,\gamma_s)-D_s$  is non-negative and $\Lambda_s=\int_0^s \lambda_u du$. For the sake of simplicity, we set $\eta_s:=\eta(s,\tilde Z_s^b+\gamma_s,\gamma_s)$ and $C$ a positive constant which may vary from line to line. Since $\tilde Y^b$ is bounded and non-positive, it holds that
\begin{align*}
&\E\left[ \left.\int_\rho^\T |e^{-\alpha \tilde Y_s^{b}}-1|^2 \lambda_s ds \right| \cal G_\rho\right] \\
&\leq C \E\left[\left. \int_\rho^\T |e^{-\alpha \tilde Y_s^{b}}-1| \lambda_s ds \right| \cal G_\rho\right]\\
& \leq C \E\left[ \left.\int_\rho^\T (-\tilde Y_s^{b}) \lambda_s ds \right| \cal G_\rho\right]\\
& \leq C \E\left[  \left.\int_\rho^\T \E\left[ \left.\int_s^T e^{-\Lambda_u}( \varphi_u+|\eta_u\cdot\tilde Z_u^b|) du \right| \cal G_s\right] e^{\Lambda_s} \lambda_s ds \right| \cal G_\rho\right]\\
&= C \E\left[\left.\int_0^{T} \E\left[\left.\textbf{1}_{\tau\geq s\geq \rho} \int_s^T e^{-\Lambda_u} (\varphi_u+|\eta_u\cdot \tilde Z_u^b|)du \; e^{\Lambda_s} \lambda_s\right| \cal G_s\right] ds\right| \cal G_\rho\right] \\
&= C \E\left[\left.\int_0^\T \textbf{1}_{s\geq \rho} \int_s^T e^{-\Lambda_u}(\varphi_u+|\eta_u\cdot \tilde Z_u^b|) du \; e^{\Lambda_s} \lambda_sds \right|\cal G_\rho\right] \\
&\leq C(E_1^\rho+E_2^\rho),
\end{align*}
where
$$E_1^\rho:= \E\left[\left.\int_\rho^\T \int_s^T e^{-\Lambda_u}\varphi_udu \; e^{\Lambda_s} \lambda_s ds\right| \cal G_\rho\right],$$
and
$$E_2^\rho:=\E\left[\left.\int_\rho^\T \int_s^T e^{-\Lambda_u}| \eta_u\cdot \tilde Z_u^b| du \; e^{\Lambda_s} \lambda_s ds \right| \cal G_\rho\right].$$
On the one hand, knowing that $\varphi$ is bounded and using the integration by part formula, we obtain
\begin{align*}
E_1^\rho&\leq C\left( \E\left[\left.\int_\rho^\T \int_s^T e^{-\Lambda_u}du \; e^{\Lambda_s} \lambda_s ds \right| \cal G_\rho\right]\right)\\
&\leq C\left( \E\left[\left.\lim\limits_{s\to \T}e^{\Lambda_s} \int_s^T e^{-\Lambda_u}du - e^{\Lambda_\rho} \int_\rho^T e^{-\Lambda_u}du  \right| \cal G_\rho\right] +(T-\rho)\right)\\
&\leq C,
\end{align*}
where $C>0$ does not depend on $\rho$. On the other hand, using the fact that $\tilde Z^b\in \H^2_{\text{BMO}}(\G)$ and from the existence of a positive constant $M'$ such that $\eta_u:=\eta(s,\tilde Z_u^b+\gamma_u,\gamma_u)\leq M'(1+\|\tilde Z_u^b\|)$, we get
\begin{align*}
E_2^\rho&= C \E\left[\left.\lim\limits_{s\to \T}e^{\Lambda_s} \int_s^T e^{-\Lambda_u}|\eta_u\cdot \tilde Z^b_u |du -e^{\Lambda_\rho} \int_\rho^T e^{-\Lambda_u}|\eta_u\cdot \tilde Z_u^b|du +\int_\rho^\T| \eta_u\cdot \tilde Z_u^b| du \right| \cal G_\rho\right]\\
&\leq C',
\end{align*}
where $C'>0$ does not depend on $\rho$. We have thus shown that under $($H2'$)$
\begin{equation}\label{bmo:sansxia} \underset{\rho \in \cal T(\G)}{\esssup}\ \E\left[ \left.\int_\rho^T |e^{-\alpha \tilde Y_s^{b}}-1|^2 \lambda_s ds \right| \cal G_\rho\right]<+\infty.\end{equation}
By considering $(\tilde{Y}^b,\tilde{Z}^b)$ the unique solution of BSDE \eqref{edsr:tilde}, previously studied, and denoting by $(Y^b,Z^b)$ the unique solution of BSDE \eqref{edsr:mainBH2'}, we know that $Y^b=\tilde{Y}^b+\xi^a_s$. So according to Inequality \eqref{bmo:sansxia}, we obtain
\begin{equation} \underset{\rho \in \cal T(\G)}{\esssup}\ \E\left[ \left.\int_\rho^T |e^{\alpha(\xi^a_s- Y_s^{b})}-1|^2 \lambda_s ds \right| \cal G_\rho\right]<+\infty.\end{equation}
Finally, under $($H2$)$ or $($H2'$)$, $\int_0^\cdot (e^{\alpha U_s}-1) dM_s$ is a {\rm BMO}$(\G)$-martingale.

\vspace{0.5em}
\noindent To conclude the proof, we have just to check that $(Y,Z,U)$ is a solution of BSDE \eqref{BSDE:Y:General:Case} in the sense of Definition \ref{Definition:Solution:BSDE:RandomHorizon} which is easily satisfied since $Y$ is bounded and $\int_0^\cdot Z_s\cdot dW_s+\int_0^\cdot (e^{\alpha U_s}-1) dM_s$ is a {\rm BMO}$(\G)$ martingale.
\end{proof}

\section{A numerical example under $($H2'$)$}
\label{numerical}
In this section, we solve numerically the exponential utility maximization problem \eqref{eq:value}. We have seen in Theorem \ref{th:expo}  that it can be reduced to solving BSDE \eqref{eq:BSDEexpo_r0}, whose solution is completely described, using Proposition \ref{prop:BroH2'}, by the solution of BSDE \eqref{edsr:mainBH2'} that we recall
$$ Y_t^{b} = \xi^a_T-\int_{t}^{T}f^b(s, Y_s^{b}, Z_s^{b}, \xi^a_s-Y_s^{b})ds-\int_{t}^{T}Z_s^{b}\cdot dW_s, ~t\in[0, T],$$
where we remind the reader that
$$f^b(s,y,z,u):=g^b(s,y,z) + \lambda_s \frac{1-e^{\alpha u}}{\alpha},\ g^b(s,y,z):=z\cdot \theta_s +\frac{\Norm{\theta_s}^2}{2\alpha}. $$
We will work for simplicity in the framework summed up in the following assumption.
\begin{assumption}\label{assum_numeric}
\begin{itemize}\text{}
\item[$(C_\xi)$] We choose $\xi^b$ in the decomposition \eqref{eq:F} equal to $0$.

\item[$(C_f)$] The coefficient $\lambda: [0, T] \rightarrow \mathbb{R}^+$ is defined by $\lambda_s=\frac{1}{T-s}$ for all $s\in [0,T]$.
\end{itemize}
\end{assumption}
\noindent Notice that
\begin{itemize}
\item Under Condition $(C_f)$, Assumption $(H2')$ is satisfied.
\item The condition $\xi^b=0$ is necessary in this paper under $($H2'$)$ in view of Proposition \ref{prop:BroH2'}.
\end{itemize}

\subsection{An implicit scheme to solve the Brownian BSDE \eqref{edsr:mainBH2'}}
In this section, we compute numerically the solution of BSDE \eqref{edsr:mainBH2'} using an implicit scheme, studied in \cite{bouchard_touzi} and \cite{bender_denk} among others, mimicking the so-called Picard iteration method to solve a Lipschitz BSDE. Our aim here is not to bring a numerical analysis of the scheme presented below, but rather to follow the method of the proof of Theorem \ref{th:mainBH2'} where the $Y$ process is obtained as a monotonic limit of solutions to Lipschitz BSDEs with $\lambda$ truncated at a level $n$. In particular, we do not prove any speed of convergence with respect to the truncation level $n$ and leave this aspect for future research. Recall the approached Lipschitz BSDE

\begin{equation} \label{edsr:lipapprox}Y_t^{b,n}=\xi_T^a -\int_t^T g^b(s,Y_s^{b,n}, Z_s^{b,n})+\lambda_s^n \frac{1-e^{\alpha(\xi^a_s-Y_s^{b,n})}}{\alpha}ds-\int_t^TZ_s^{b,n}\cdot dW_s, \end{equation}
with $g^b(s,y,z)=\frac{\|\theta_s\|^2}{2}+ \theta_s\cdot z$ and $\lambda_s^n:= \lambda_s \wedge n$.

\vspace{0.5em}
\noindent
Let $(t_k)_k$ be a subdivision of $[0,T]$ such that $0=t_0<t_1<...<t_N=T$, and denote by $\Delta_k$ the increment $t_{k+1}-t_k$. For the sake of simplicity, we also introduce the notation $\Delta_k W:= W_{t_{k+1}}-W_{t_k}$.  Denoting by $(Y^{b,n,L},Z^{b,n,L})$ the solution to the Lth Picard iteration associated to \eqref{edsr:lipapprox}, the solution of BSDE \eqref{edsr:lipapprox} associated to a truncation level $n$ is computed by
\begin{equation}\label{schema}
\begin{cases}
\displaystyle Y^{b,n,L}_T=\xi^a_T,\\
\displaystyle Z^{b,n,L}_{t_k}=\frac{1}{\Delta_k}\E\left[ Y^{b,n,L}_{t_{k+1}} \Delta_k W\right],\\
\displaystyle Y^{b,n,L}_{t_{k}}=\E\left[Y^{b,n,L}_{t_{k+1}}\Big| \mathcal{F}_{t_k}\right] -\Delta_k\left( g^{b}(t_{k}, Y^{b,n,L-1}_{t_{k}}, Z^{b,n,L-1}_{t_{k}})+ \lambda_{t_{k}}\wedge n \frac{1-e^{\alpha (\xi^a_{t_{k}}-Y^{b,n,L-1}_{t_{k}})}}{\alpha}\right).
 \end{cases}
 \end{equation}
 In all this section, we assume that the increment $\Delta_k$ is constant, and we set $\Delta:=\Delta_k$.

 \begin{remark}
 Notice that the truncation does not act as soon as $n\geq 1/\Delta$. So, this numerical scheme limits us to choose $n$ smaller than $1/\Delta$. Obviously, when $\Delta$ goes to $0$, the truncation acts for bigger truncation level $n$. So, limiting $n$ to be smaller than $\Delta$ is in fact an artifact of the computation coming from the previous numerical scheme.
 \end{remark}
\subsection{Numerical solution of the utility maximization problem \eqref{eq:pb}}

In this section, we solve numerically the utility maximization problem \eqref{eq:pb} when $d=1$ for simplicity.  We need to build a default time $\tau$ knowing that its associated intensity $\lambda$ is given by Relation \eqref{eq:defM}. According to \cite{jeanblanc_song}, given a positive $\mathbb{G}$-local martingale and an increasing process $\Lambda$ such that $Z_t:=N_t e^{-\Lambda_t}\leq 1$, for $t\geq 0$, we can construct a probability measure $\Q^Z$ such that $\Q^Z(\tau>t)=Z_t$. In particular, taking $N\equiv 1$, from \cite[Section 2.1]{jeanblanc_song}, $\tau$ is an exponential random variable with intensity $\lambda$. Then, by setting $\phi$ an exponential random variable with intensity $1$, the default time $\tau_n$ associated with intensity $\lambda\wedge n$ is given by
\begin{equation}\label{taun}
\displaystyle  \tau^n=\inf\left\{ t\geq0,  \  \int_0^t \lambda_s\wedge n \; ds \geq \phi \right\}\wedge T.
\end{equation}
Notice that $(\tau^n)_n$ is a non-decreasing bounded sequence, which converges to $\tau$ defined by
$$ \tau=\inf\left\{ t\geq 0,  \  \int_0^t \lambda_s \; ds \geq \phi \right\}\wedge T.$$

\begin{prop}
Under Assumption \ref{assum_numeric}, Hypothesis $($H1$)$ holds for every $\tau_n$.
\end{prop}
\begin{proof} This result is a direct consequence of \cite[Section 12.3.1]{filipovic}.
\end{proof}

\noindent We give now an explicit formula to compute $\tau_n$. According to \eqref{taun}, $\tau_n$ satisfies the following equation for $\phi$ an exponential random variable
$$ \int_0^{\tau_n} \frac{1}{T-s}\wedge n\; ds =\phi.$$
By considering the two cases $s\leq T-\frac1n$ and $s\leq T-\frac1n$ we get
\begin{align*}
\phi&= \int_0^{\tau_n \wedge (T-\frac1n)} \frac{1}{T-s}\; ds + \int_{\tau_n \wedge (T-\frac1n)}^{\tau_n} n\; ds\\
&=\log\left(\frac{T}{T-\tau_n \wedge (T-\frac1n)}\right) + n\left(\tau_n-\tau_n \wedge \left(T-\frac1n\right)\right).
\end{align*}
If $\tau_n \leq T-\frac1n$, then $\tau_n = T(1-e^{-\phi})$ and if $\tau_n> T-\frac1n$ then $\tau_n=\frac{\phi+nT-1-\log(nT)}{n}$. Thus, the simulation of $\tau_n$ can be easily achieved from the simulation of the exponentially distributed random variable $\phi$.

\vspace{0.5em}
\noindent Assume that when the default time appears before the maturity $T$, the agent has to buy a put with strike $K$. Then, $\xi^a$ is given by
\begin{equation}\label{xia_put} \xi^a_s:= \left(K-S_0e^{\sigma W_s + \left(\mu-\frac{\sigma^2}{2}\right)s}\right)^+.\end{equation}
From now on, we use the following data\vspace{0.5em}

\noindent \textbf{Data.} $T=1, \; \alpha=0.25,\; \Delta=0.02,\; S_0=0.5,\; \sigma=1.0,\; \mu=1.0,\; K=1.0,\; \theta=1.0$. We take three truncation level $n_1=50,\; n_2=10,\; n_3=2, n_4=1$ and we simulate $M=10^6$ paths of the solution $(Y^{b,n_i}, Z^{b,{n_i}})$ for $i\in \{ 1,2,3, 4\}$. Note that as $\Delta=0.02$ any truncation level $n$ greater than $50$ is pointless by Assumption ($C_f$). Then, we obtain
 \begin{center}
 \[
\begin{array}{|c|c|c|c|}
\hline n &\tau^n   & \xi^a_{\tau^n}  & Y_0^n\\
\hline 50 &0.562075&0.337748&2.40391\\
\hline 10 &0.562075&0.337748 & 1.31611 \\
 \hline 2 &0.56628&0.336354 & 0.01315\\
 \hline 1 & 1& 0.175639 & -0.519817\\
\hline\end{array}
\]

 \end{center}
\noindent The same path of the solutions of BSDE \eqref{edsr:lipapprox} for a truncation level $n_i$, for $i\in \{1,2,3,4\}$, denoted $(Y^{b,n_i},Z^{b,n_i})$ are given in Figure \ref{edsr:brow}.
\begin{figure}\caption{\label{edsr:brow}  Solutions of BSDE \eqref{edsr:lipapprox} with truncation levels $n_1=50,$ $n_2=10,$ $n_3=2,$ $n_4=1$ and $n=0$ with $Y_0^0=-1.37$.}
\begin{center}
 \includegraphics[scale=0.7]{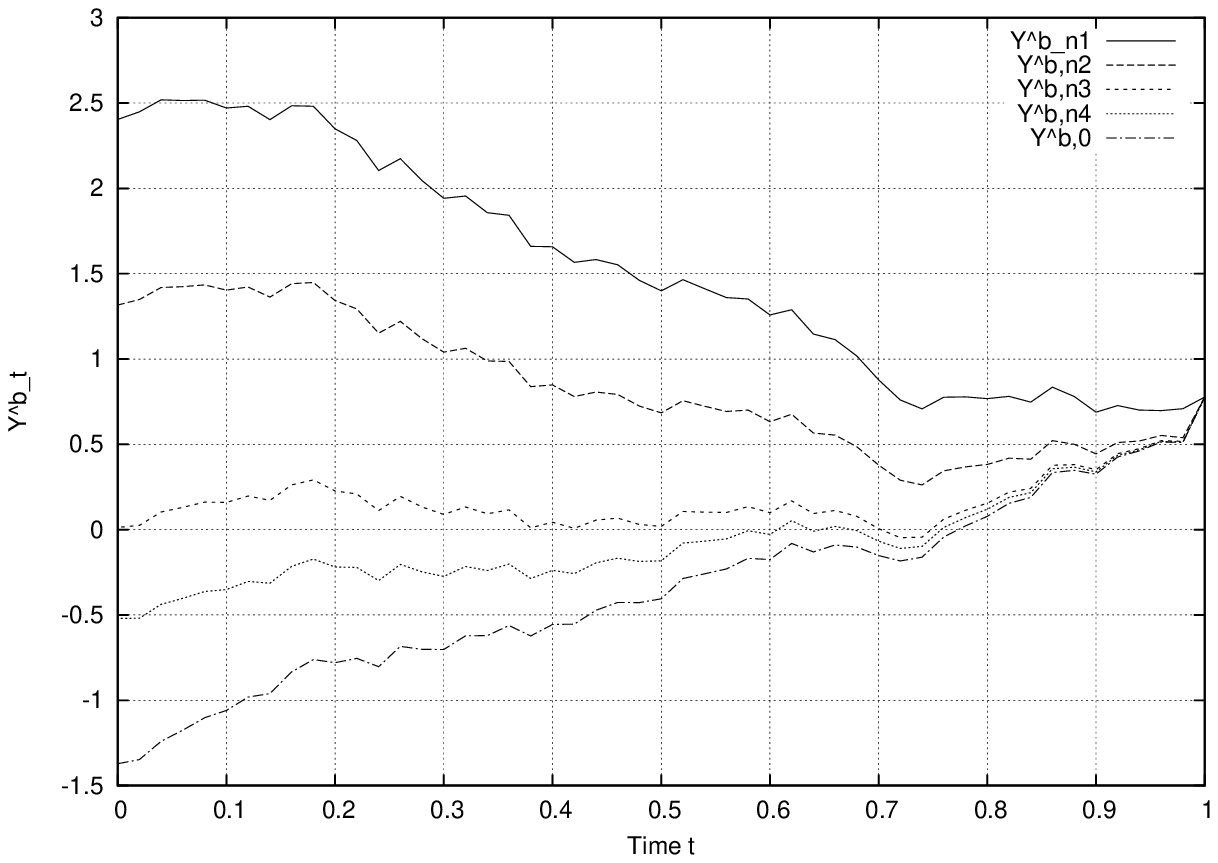}

  \includegraphics[scale=0.7]{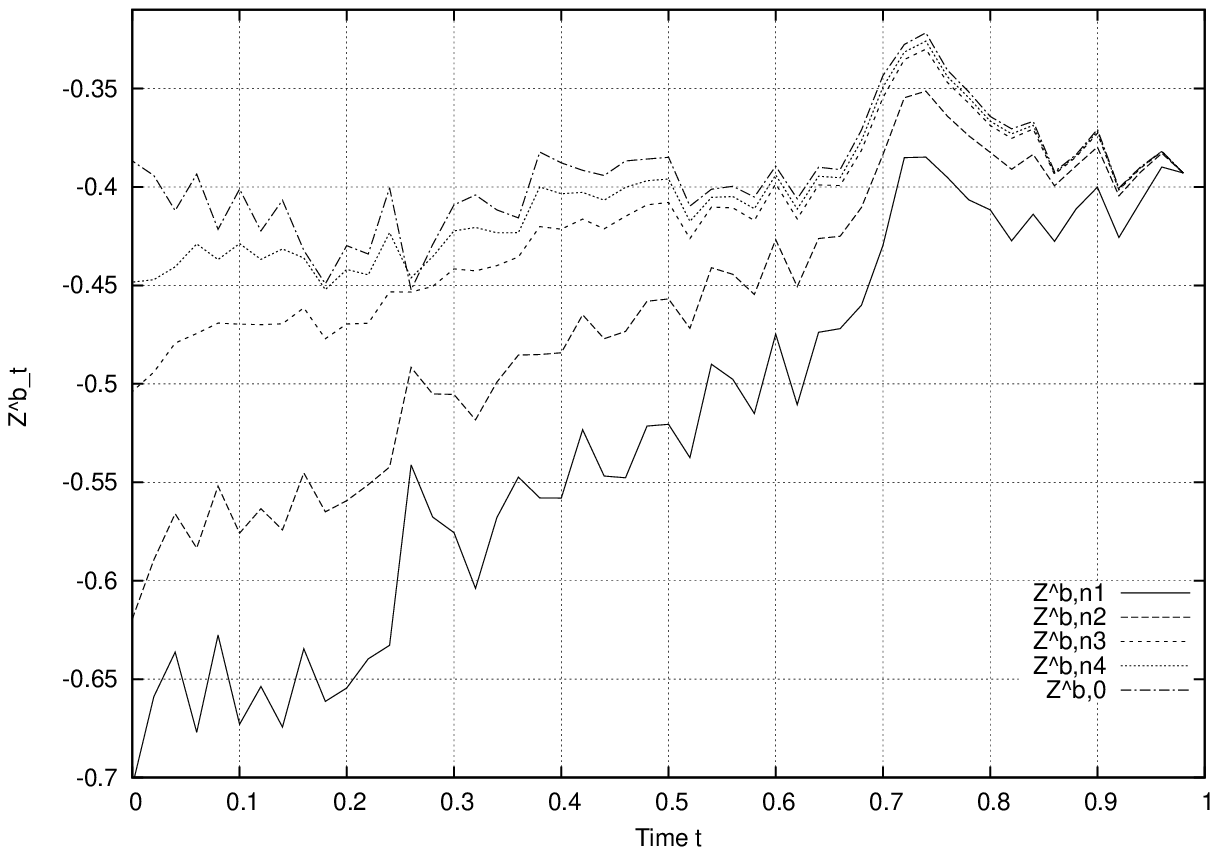}
   \end{center}
 \end{figure}

\noindent Given a truncation level $n$, we would like emphasize the dependence between the probability that the default time appears after $T$ and the value of the utility maximization problem \eqref{eq:pb}. Denote $p^n:=\mathbb{P}(\tau^n>T)$ and notice that $p^n$ is non-increasing with respect to $n$ since $(\tau^n)_n$ is non-decreasing. According to \cite{Jeulin}
$$p^n=e^{-\int^T_0 \lambda_s \wedge n \; ds}.$$
We can compute easily $p^n$ as a function of $n$ by considering the cases $T \leq \frac1n$ and $T>\frac1n$. Then we obtain
$$ p^n=
\begin{cases}
e^{-nT} & \text{ if }  T\leq \frac1n\\
\frac{e^{-1}}{nT}  &\text{ if }  T> \frac1n
\end{cases}.
$$
Besides, the case $n=0$ corresponds to the classical utility maximization problem without default time. Moreover, we know that $\lim\limits_{n \to +\infty} \tau^n=\tau$ and recall that under Assumption (H2'), the support of $\tau$ is $[0,T]$ we obtain $\lim\limits_{n\to +\infty} p^n=0$. The value $V^n(1)$ of the utility maximization problem \eqref{eq:pb} associated to the default time $\tau_n$ is given by $V^n(1):=-e^{-\alpha(1-Y_0^{b,n})}$. Since $p^n$ (resp. $Y^n_0$) is non-increasing (resp. non-decreasing) with respect to $n$, $V^n(1)$ is a non-increasing function of $n$ and thus $V^n(1)=F(p^n)$ with $F: [0,1] \longrightarrow \mathbb{R}^-$ a non-decreasing mapping. \vspace{0.5em}

\begin{figure}\centering\caption{\label{valeur_pn}$V^n(1)$ as a function of $p^n$, $n\in \{0, \dots ,50\}$.}
\begin{center}\includegraphics[scale=0.7]{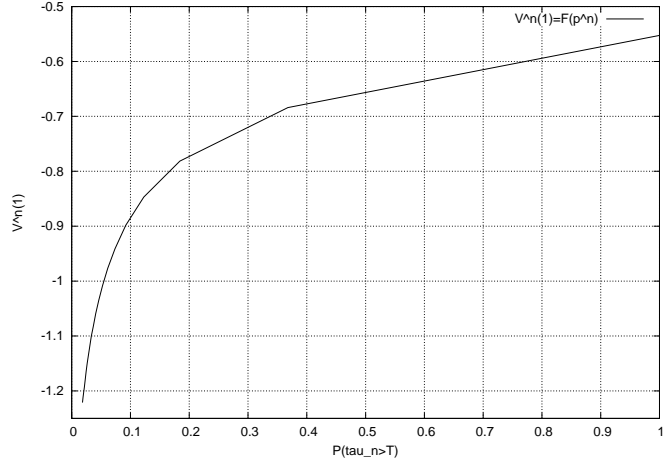} \end{center}
\end{figure}

\paragraph{Interpretation of Figure \ref{valeur_pn}} When there is a default time, which corresponds to the case $n\to +\infty$, the value of Problem \ref{eq:pb} is obviously less than the case without default time (which corresponds to $n=0$). We can interpret this by the fact that the performance of the investor when she knows that her default time appears before the maturity is less than her performance in the case without default time.

\vspace{0.5em}
 \noindent We now study the influence of $p^n$ on the indifference price of the claim $\xi$, denoted by $P_n$. Recall that: $$ P_n:=\inf\left\{ p\geq 0, \; V(x+p) \geq V^0(x)\right\},$$
where $V^0$ corresponds to the value of Problem \ref{eq:pb} when $\xi \equiv 0$. We denote by $(y^{b,n},z^{b,n})$ the unique solution to BSDE \eqref{edsr:lipapprox} when $\xi \equiv 0$:
\begin{equation}\label{edsr:xi0}y^{b,n}_t=0-\int_t^T  \lambda_s^n \frac{1-e^{-\alpha y^{b,n}_s}}{\alpha}+\frac{|\theta_s|^2}{2}+\theta_s z_s^{b,n} ds-\int_t^T z^{b,n}_s dW_s.\end{equation}

\noindent We deduce that $P_n$ satisfies
\begin{align*}
&V(x+P_n)=V^0(x)\\
\Longleftrightarrow& -e^{-\alpha(x+P_n-Y_0^{b,n})}=-e^{-\alpha(x-y_0^{b,n})}\\
\Longleftrightarrow& P_n=Y_0^{b,n}-y_0^{b,n}.
\end{align*}

\begin{prop}
$P_n$ is a non-negative and non-increasing function $G$ of $p_n$.
\end{prop}
\begin{proof} Denote by $(\y^n, \z^n)$ a pair of adapted processes defined by $\y^n_t:= Y^{b,n}_t-y^{b,n}_t$, and $\z^{b,n}_t:=Z_t^{b,n}-z_t^{b,n}$ for $t\in [0,T]$ where $(Y^{b,n}, Z^{b,n})$ (resp. $(y^{b,n},z^{b,n})$) is the unique solution of BSDE \eqref{edsr:lipapprox} (resp. \eqref{edsr:xi0}). Then, $(\y^n, \z^n)$ is the unique solution of the following (Lipschitz) BSDE
$$\y^n_t =\xi^a_T -\int_t^T \lambda_s^n \frac{e^{-\alpha y_s^{b,n}}-e^{\alpha(\xi^a_s-Y_s^{b,n})}}{\alpha} + \theta_s \z_s^n ds -\int_t^T \z_s^n dW_s, $$ which can be rewritten, using the mean value theorem, as
$$\y^n_t =\xi^a_T -\int_t^T \lambda_s^n e^{\alpha \overline{Y}_s^n}(\y_s^n-\xi_s^a)  + \theta_s \z_s^n ds -\int_t^T \z_s^n dW_s, $$
with $\overline{Y}_s^n$ a bounded adapted process between $\xi^a_s-Y_s^{b,n}$ and $y^{b,n}_s$ for $s\in [t,T]$. From the comparison Theorem for Lipschitz BSDEs and since $\xi^a$ given by \eqref{xia_put} is a non-negative process,  we deduce that $\y^n$ is non-decreasing in $n$.  Thus $P^n:=\y_0^n$ is a non-increasing mapping of $p^n$. Besides, by noticing that
$$Y^0_s=\E^\mathbb{Q}[\xi^a_T | \mathcal{F}_s], \; \frac{d\mathbb{Q}}{d\mathbb{P}}:=\mathcal{E}\left( -\int_0^T \theta_s  dW_s\right),$$ we deduce that $P^n:=\y^n_0\geq \y^0_0\geq 0$ for all $n$.
\end{proof}

\noindent We now compute $P^n=G(p^n)$ in Figure \ref{indiff_price}.
 \begin{figure}\centering\caption{\label{indiff_price} Indifference price $P_n$ as a function of $p^n$, $n\in \{0, \cdots , 50\}$.}
\begin{center}\includegraphics[scale=0.7]{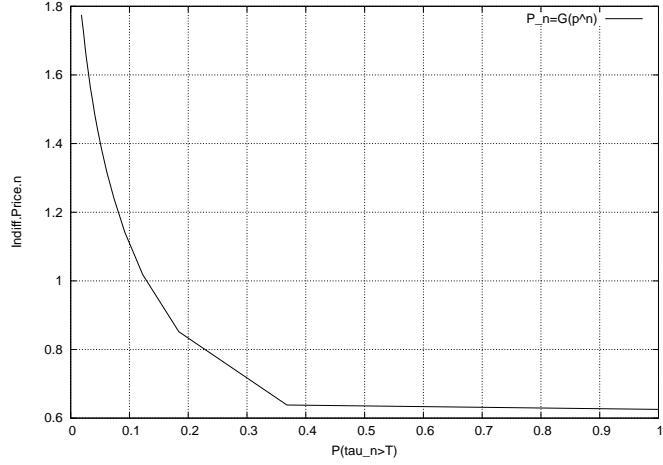}\end{center}
\end{figure}
 \paragraph*{Some remarks concerning Figure \ref{indiff_price}}
  \begin{itemize}
\item $P_n$ seems to be a non-convex function of $p^n$.
\item When $n=0$ (\textit{i.e.} $p^n=1$), we get $P_0= Y_0^0-y_0^0$. Note that $(y^0,z^0)$ is the unique solution of the following BSDE
$$y^0_t= 0-\int_t^T z^0_s \theta + \frac{\theta^2}{2\alpha}ds -\int_t^T z_s dW_s. $$ The (unique) solution is given by $y^0_t=\frac{-\theta^2}{2\alpha}(T-t)$ and $z^0_t=0$ for $t\in [0,T]$.
\end{itemize}

\vspace{0.5em}
\noindent Now, we denote by $(Y,Z, U)$ the solution of the following BSDE
 \begin{equation}\label{edsr:maintrun}
Y_t = \xi -\int_\t^\T Z_s dW_s -\int_\t^\T U_s dH_s - \int_\t^\T f(s,Y_s,Z_s,U_s) ds, \  t\in [0,T].
\end{equation}
Then, from Proposition \ref{prop:BroH2'},
 \begin{align*}
Y_t &= Y_t^{b} \textbf{1}_{t<\tau} + \xi^a_{\tau}\textbf{1}_{t\geq\tau},\\
Z_t &= Z^{b}_t \textbf{1}_{t\leq\tau},\\
U_t &= (\xi^a_t-Y_t^{b}) \textbf{1}_{t\leq\tau}.
\end{align*}
 Recall that this BSDE solves the utility maximization problem \eqref{eq:pb} through the $Y$ and the $Z$ components. We give numerically a path of this BSDE in Figure \ref{figure:yz}, obtained by computing $\tau(\omega) =0.562075$ with $\omega \in \Omega$.
  \begin{figure}\caption{\label{figure:yz}Components $Y,Z$ of the solution of BSDE \eqref{edsr:maintrun}.}
 \begin{center}
 \includegraphics[scale=0.7]{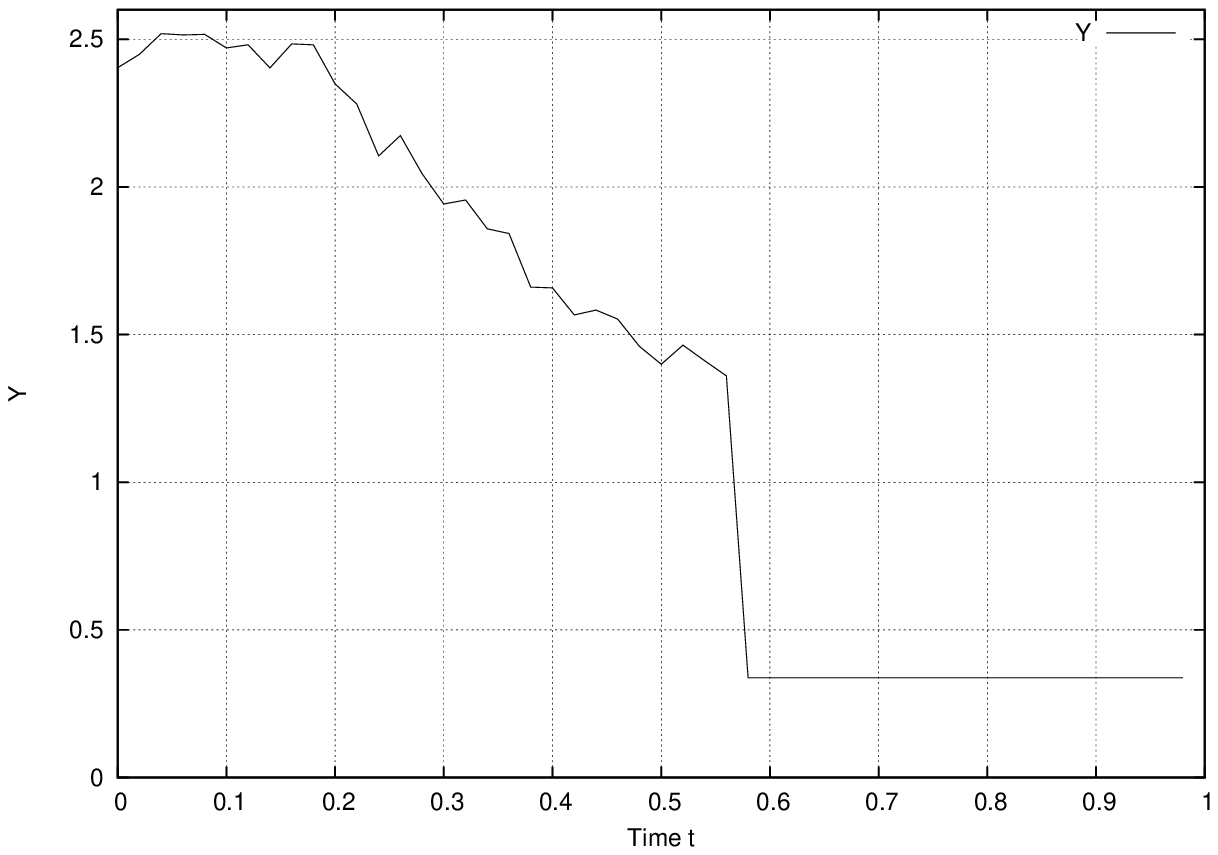}

 \includegraphics[scale=0.7]{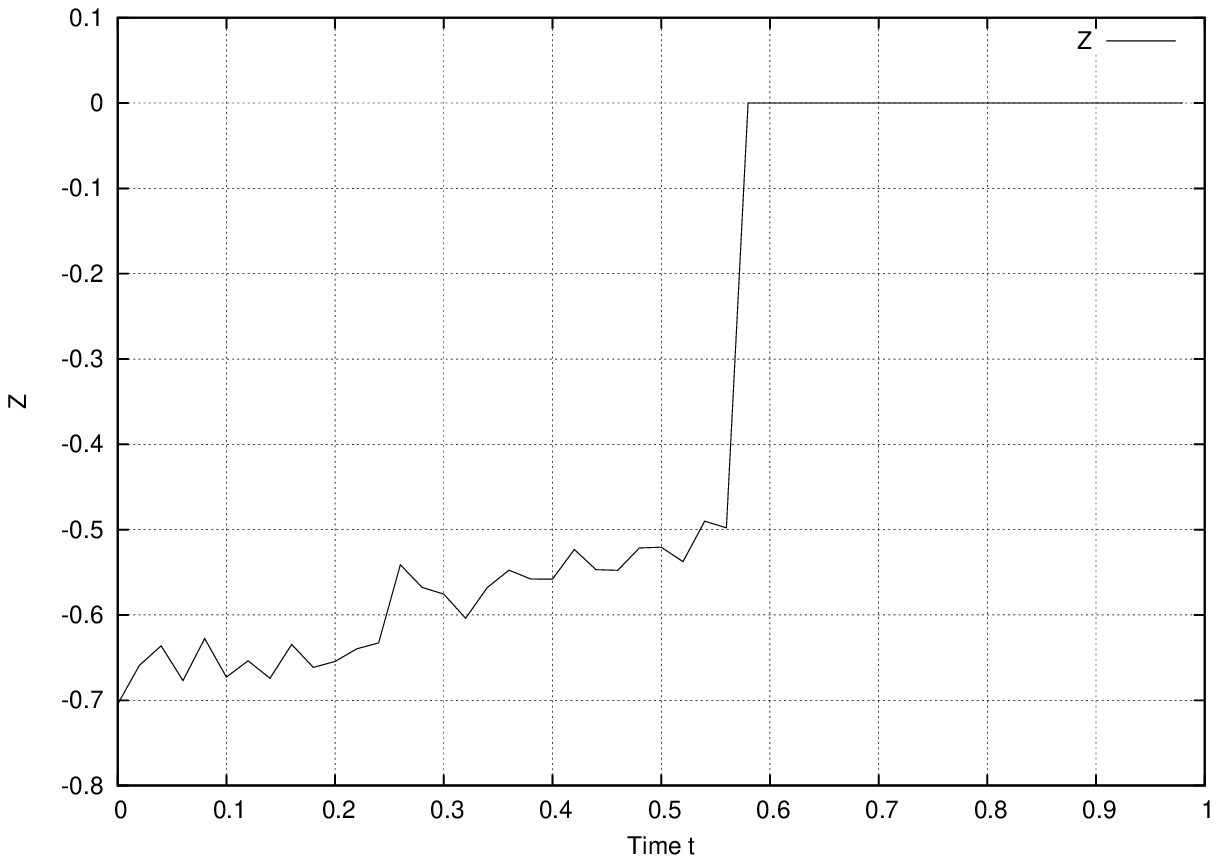}
 \end{center}
 \end{figure}

\noindent According to Theorem \ref{th:expo}, an optimal strategy $p^*$ is given by $p^*=(Z_t+\frac\theta\alpha)\mathbf{1}_{t\leq \tau}$. We compute an optimal strategy to Problem \eqref{eq:pb} in Figure \ref{fig:solpb} associated to an initial wealth $x=1$ and we compare it with the classical case without jump.

\begin{figure}\begin{center}\begin{center}\caption{\label{fig:solpb} An optimal strategy associated to the exponential utility maximization problem \eqref{eq:pb} with $\omega$ such that $\tau(\omega) =0.562075$ and without default time.}\end{center}
  \includegraphics[scale=0.7]{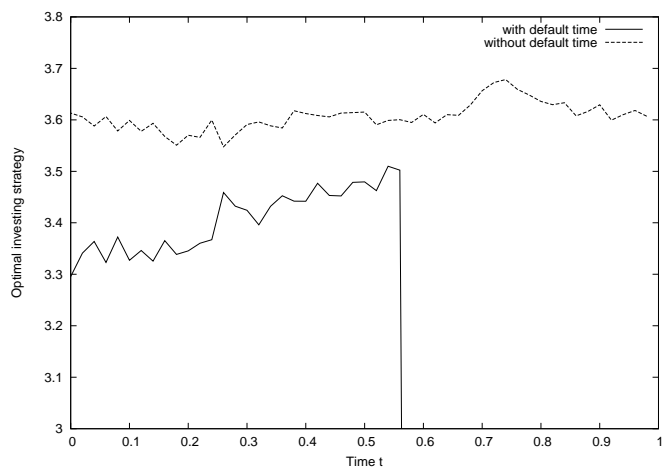}
    \end{center}
\end{figure}
\vspace{0.5em}

\paragraph{Interpretation of Figure \ref{fig:solpb}} In this very particular case, when we assume that the default time $\tau$ appears almost surely before the maturity, the investor tends to be more cautious by investing less in the risky asset. It is quite reasonable since she knows that she will pay $\xi^a_\tau$ which is a non-negative random variable at default. Note that contrary to what happens for small times where the trading strategies are merely mirrors of each other, the strategy in the default problem becomes more and more similar to the one in the non-default case and the former tends to coalesce with the latter.

\section*{Acknowledgments}
The authors thank an Associate Editor and a Referee for their careful reading of this paper and their suggestions.

\end{document}